\documentclass[letterpaper]{article}
\usepackage{amsthm}
\usepackage{amssymb,latexsym,graphics,enumerate,bbm}
\usepackage[mathscr]{eucal}
\usepackage{amsmath,amsfonts,amsthm,amssymb,mathtools}
\usepackage{color}
\usepackage{hyperref}
\usepackage[left=1in,right=1in,top=1in,bottom=1in]{geometry}

\usepackage[abbrev,lite,nobysame]{amsrefs}

\usepackage{graphicx}
\numberwithin{equation}{section}
 
\numberwithin{equation}{section}

\newcommand{\rr}{\mathbb{R}}

\newcommand{\lan}{\langle}
\newcommand{\ran}{\rangle}
\newcommand{\be}{\begin{eqnarray*}}
\newcommand{\bel}{\begin{eqnarray}}
\newcommand{\ee}{\end{eqnarray*}}
\newcommand{\eel}{\end{eqnarray}}
\newcommand{\ba}{\begin{aligned}}
\newcommand{\ea}{\end{aligned}}
\newcommand{\de}{\Delta}
\newcommand{\al}{\alpha}
\newcommand{\na}{\nabla}
\newcommand{\ep}{\epsilon}
\newcommand{\nq}{{\neq}}

\newcommand{\myr}[1]{{ #1 }}

\newcommand{\pa}{\partial}
\newcommand{\wh}{\widehat}

\newcommand{\lf}{\left}
\newcommand{\rg}{\right}

\newtheorem{thm}{Theorem}

\newtheorem{lem}{Lemma}

\newtheorem{rmk}{Remark}

\numberwithin{thm}{section}
\numberwithin{rmk}{section}
\numberwithin{lem}{section}

\newcommand\R{{\mathbb R}}
\newcommand\T{{\mathbb T}}

\newcommand\Torus{{\mathbb T}}

\newcommand{\cF}{\mathcal{F}}
\newcommand{\cG}{\mathcal{G}}

\mathtoolsset{showonlyrefs}
\allowdisplaybreaks

\newcommand{\n}{\ensuremath{\nonumber}}
\title{A Note on Enhanced Dissipation and Taylor Dispersion of Time-dependent Shear Flows}
\author{Daniel Coble\thanks{\footnotesize  University of South Carolina, Columbia, SC 29208, USA  \href{mailto:dncoble@email.sc.edu}{\texttt{dncoble@email.sc.edu}}} \and Siming He \thanks{Department of Mathematics, University of South Carolina, Columbia, SC 29208, USA \href{mailto:siming@mailbox.sc.edu}{\texttt{siming@mailbox.sc.edu }}\\
\textbf{Acknowledgment.} This work was  partly  supported by NSF grants DMS 2006660, DMS 2304392, DMS 2006372.} }

\begin{document}
    \maketitle 
\begin{abstract}
This paper explores the phenomena of enhanced dissipation and Taylor dispersion in solutions to the passive scalar equations subject to time-dependent shear flows. The hypocoercivity functionals with carefully tuned time weights are applied in the analysis. We observe that as long as the critical points of the shear flow vary slowly, one can derive the sharp enhanced dissipation and Taylor dispersion estimates, mirroring the ones obtained for the time-stationary case.
\end{abstract}
\tableofcontents

    \section{Introduction}\label{introduction}
    In this paper, we consider the passive scalar equations
    \begin{align}\label{PS}
    \pa_t f+V(t,y)\pa_x f=\nu \de_\sigma f,\quad f(t=0,x,y)=f_0(x,y). 
    \end{align}
    Here $f$ denotes the density of the substances, and $(V(t,y),0)$ is a time-dependent shear flow. The P\'eclet number $\nu>0$ captures the ratio between the transport and diffusion effects in the process. Here $\de_\sigma=\sigma\pa_{x}^2+\pa_y^2,\, \sigma\in\{0,1\}.$ We consider three types of domains: $\Torus\times \rr,\, \Torus^2, \, \rr\times [-1,1]$. The torus $\Torus$ is normalized such that $ \Torus =[-\pi,\pi].$
    
    In recent years, much research has been devoted to studying enhanced dissipation and Taylor dispersion phenomena associated with the equation \eqref{PS} in the regime $0<\nu\ll1.$ To understand these phenomena, we first identify the relevant time scale of the problem. The standard $L^2$-energy estimate yields the following energy dissipation equality: 
    \begin{align}
    \frac{d}{dt}\|f\|_{L^2}^2=-2\nu \sigma\|\pa_x f\|_{L^2}^2-2\nu\|\pa_y f\|_{L^2}^2.\label{engy_rel}
    \end{align}
    Hence, at least formally, we expect that the energy ($L^2$-norm) of the solution decays to half of the original value on a long time scale $\mathcal{O}(\nu^{-1}).$ This is called the ``heat dissipation time scale''. However, a natural question remains: since the fluid transportation can create gradient growth of the density $\na f$, which makes the damping effect in \eqref{engy_rel} stronger, can one derive a better decay estimate of the solution to \eqref{PS}? This question was answered by Lord Kelvin in 1887 for a special family of flow $V(t,y)=y$ (Couette flow) \cite{Kelvin87}. He could explicitly solve the equation \eqref{PS} and read the exact decay rate through the Fourier transform. To present his observation, we first restrict ourselves to the cylinder $\Torus\times \rr$ or torus $\Torus^2$ and define the concepts of horizontal average and remainder:
    \begin{align}\label{x_av_rm}
    \lan f\ran(y) =\frac{1}{2\pi}\int_{-\pi}^{\pi} f(x,y)dx,\quad f_\nq(x,y)=f(x,y)-\lan f\ran(y).
\end{align}
We observe that the $x$-average $\lan f\ran$ of the solution to \eqref{PS} is also a solution to the heat equation. Hence it decays with rate $\nu$. On the other hand, the remainder $f_\nq$ still solves the passive scalar equation \eqref{PS} with $f_\nq(t=0,x,y)=f_{0;\nq}(x,y)$ and something nontrivial can be said. Lord Kelvin showed that there exists constants $C,\,\delta>0$ such that the following estimate holds 
\begin{align}\label{ed_mon_sharp}
\|f_{\nq}(t)\|_{L^2}\leq C\|f_{0;\nq}\|_{L^2}e^{-\delta\nu^{1/3}t},\quad \forall t\geq 0. 
\end{align} 
One can see that significant decay of the remainder happens on time scale $\mathcal{O}(\nu^{-1/3})$, which is much shorter than the heat dissipation time scale. This phenomenon is called the \emph{enhanced dissipation}. 

However, new challenges arise when one considers shear flows different from the Couette flow. In these cases, no direct Fourier analytic proof is available at this point. We focus on two families of shear flows, i.e., strictly monotone shear flows and nondegenerate shear flows. In the paper \cite{BCZ15}, J. Bedrossian and M. Coti Zelati apply hypocoercivity techniques to show that for \emph{stationary} strictly monotone shear flows $\{(V(y),0)|\inf|V'(y)|\geq c>0,\, y\in \rr\}$, the following estimate is available
\begin{align*}
\|f_\nq(t)\|_{L^2}\leq C \|f_{\nq;0}\|_{L
^2}e^{-\delta\nu^{1/3}|\log\nu|^{-2}t},\quad \forall t\geq 0.
\end{align*}
Later on, D. Wei applied resolvent estimate techniques to improve their estimate to \eqref{ed_mon_sharp} \cite{Wei18}. 

When we consider non-constant smooth shear flows on the torus $\Torus^2$, an important geometrical constraint has to be respected, namely, the shear profile $V$ must have critical points $\mathcal{C}:=\{y_\ast|\pa_y V(y_\ast)=0\}$. Nondegenerate shear flows are a family of shear flows such that the second derivative of the shear profile does not vanish at these critical points, i.e., $\min_{y_\ast\in \mathcal{C}}|\pa_{y}^2V(y_\ast)|\geq c>0$. In the papers \cites{BCZ15, Wei18}, it is shown that if the underlying shear flows are  \emph{stationary} and non-degenerate, there exist constants $C\geq 1,\ \delta>0$ such that
\begin{align}\label{ED_non_deg}
\|f_{\nq}(t)\|_{L^2}\leq C\|f_{0;\nq}\|_{L^2}e^{-\delta\nu^{1/2}t},\quad \forall t\in[0,\infty).
\end{align}
    In the paper \cite{CotiZelatiDrivas19}, it is shown that the enhanced dissipation estimates \eqref{ed_mon_sharp}, \eqref{ED_non_deg} are sharp for \emph{stationary} shear flows. In the paper \cites{CKRZ08, ElgindiCotiZelatiDelgadino18,FengIyer19}, the authors rigorously justify the relation between the enhanced dissipation effect and the mixing effect. In the paper \cite{AlbrittonBeekieNovack21}, the authors apply H\"ormander hypoellipticity technique to derive the estimates \eqref{ed_mon_sharp}, \eqref{ED_non_deg} on various domains. Further enhanced dissipation in other flow settings, we refer the interested readers to the papers \cites{He21,FengFengIyerThiffeault20,CotiZelatiDolce20},
and the references therein. The enhanced dissipation effects have also found applications in many different areas, ranging from hydrodynamic stability to plasma physics, we refer to the following papers \cite{BMV14,BGM15I,BGM15II,BGM15III,ChenLiWeiZhang18,BedrossianHe16,BedrossianHe19,GongHeKiselev21,HeKiselev21,BedrossianBlumenthalPunshonSmith192, WeiZhang19,LiZhao21,CotiZelatiDietertGerardVaret22,AlbrittonOhm22,Bedrossian17, He,HuKiselevYao23,HuKiselev23,KiselevXu15,IyerXuZlatos,CotiZelatiDolceFengMazzucato,FengShiWang22}.     

The enhanced dissipation can be rigorously justified for the periodic domains $\Torus^2,\ \Torus\times \rr$. Based on these observations, one might ask whether extending these results to infinitely long channels, e.g., $\rr\times[-1,1]$ is possible. The question is highly nontrivial. As mentioned above, the enhanced dissipation phenomenon is closely related to the mixing effect associated with the fluid field, \cite{ElgindiCotiZelatiDelgadino18,FengIyer19}. However, it is widely recognized that the mixing effect can be weak in infinite channels. As it turns out, in the infinite channel, another fluid transport effect -   \emph{Taylor dispersion} - plays an important role, see, e.g., \cite{Aris56,Taylor53,YoungJones91}. For a mathematically rigorous justification, we refer to the papers, \cite{CotiZelatiGallay21,BedrossianCotiZelatiDolce22,BeckChaudharyWayne20,CotiZelatiDolceLo23
}. It is also worth mentioning that the
Taylor dispersion is also related to the \emph{quenching phenomenon} of shear flows, see, e.g., \cite{CKR00,KiselevZlatos,HeTadmorZlatos,Zlatos11 , Zlatos2010}.

Most of the results we present thus far are centered around \emph{stationary} flows. In this paper, we focus on \emph{time-dependent} shear flows and hope to identify sufficient conditions that guarantee enhanced dissipation and Taylor dispersion. Before stating the main theorems, we provide some further definitions. After applying a Fourier transformation in the $x$-variable, we end up with the following $k$-by-$k$ equation
\begin{align}\label{k_by_k_eq}
\pa_t \wh f_k(t,y)+V(t,y)ik \wh f_k(t,y)=\nu\pa_{y}^2 \wh f_k-\sigma\nu|k|^2 \wh f_k(t,y),\quad \wh f_k(t=0,y)= \wh f_{0;k}(y). 
\end{align}
We will drop the $\wh{(\cdot)}$ notation later for simplicity.  
The main statements of our theorems are as follows:

\begin{thm}\label{thm_mon}
Consider the solution to the equation \eqref{k_by_k_eq} initiated from the initial data $f_0\in C_c^\infty(\Torus\times\rr)$. Assume that on the time interval $[0,T]$, the $C_t C^2_{y}$ velocity profile $V(t,y)$ satisfies the following constraint   
\begin{align}
\inf_{t\in[0,T],\ y\in\rr}|\pa_y V(t,y)|\geq \mathfrak{c}>0,\quad \|V\|_{L_t^\infty([0,T]; W_y^{3,\infty})}<C.\label{V_y}
\end{align}
Then there exists a threshold $\nu_0(V)$  such that for $\nu<\nu_0$, the following estimate holds
\begin{align}\label{mon_ED_main}
\|f_{k}(t)\|_{L^2}\leq e\| f_{0;k}\|_{L^2}\exp\lf\{-\delta\nu^{1/3}\myr{|k|^{2/3}} t\rg\},\quad \forall t\in [0,T].
\end{align}
Here $\delta>0$ are constants depending only on the parameter $\mathfrak{c}$ and $\|V\|_{L_t^\infty C_y^3}$ \eqref{delta_mono}.
\end{thm}

The next theorem is stated as follows.

\begin{thm}\label{ndeg_main_thm} Consider the solution to the equation \eqref{PS} initiated from the smooth initial data $f_0\in C^\infty(\Torus^2)$. Assume that the shear flow $V(t,y)\in C^2_{t,y}$ satisfies the following structure assumptions on the time interval $[0,T]$: 

a) Phase assumption: There exists a nondegenerate reference shear $U\in C_t^1C_y^2$ such that the time-dependent flow $V(t,y)$ and the reference flow $U(t,y)$ share all their nondegenerate critical points $\{y_i(t)\}_{i=1}^N$, where $N$ is a fixed finite number. Moreover,
\begin{align}
\pa_y V(t,y)\pa_yU(t,y)\geq &0,\ \ \quad\forall y\in \Torus,\ \forall t\in[0,T], \\
\|\pa_{ty}U\|_{L_t^\infty([0,T];L^\infty_y)}\leq \nu^{3/4},\quad\|V&\|_{L_t^\infty([0,T]; W_y^{2,\infty})}+\|U\|_{L_t^\infty([0,T];W^{2,\infty})}<C.   \label{asmpt0}
\end{align}

b) Shape assumption: 
  there exist $N$ pairwise disjoint open neighborhoods $\{B_{r}(y_i(t))\}_{i=1}^N$ with fixed radius $0<r=\mathcal{O}(1)$, and two constants $\mathfrak{C}_0,\ \mathfrak{C}_1> 1$ such that the following estimates hold for $Z(t,y)\in \{V(t,y), U(t,y)\}$,
\begin{align}\label{asmp_1}
 \mathfrak{C}_0^{-1}(y-y_i(t))^2\leq |\pa_yZ|^2\leq& \mathfrak{C}_0(y-y_i(t))^2, \quad \mathfrak{C}_0>0, \quad \forall y\in B_{r}(y_i(t));\\ 
0<\mathfrak{C}_1^{-1}\leq  |\pa_y Z|\leq& \mathfrak{C}_1 ,\quad \forall y\notin \cup_{i=1}^N B_{r}(y_i(t)),\label{asmpt2} 
\end{align}
Then there exists a threshold $\nu_0(U,V)$ such that if $\nu\leq \nu_0$, the following estimate holds\begin{align}
\|f_{k}(t)\|_{L^2}\leq e\|f_{k}(0)\|_{L^2}\exp\lf\{-\delta\nu^{1/2}\myr{|k|^{1/2}} t\rg\},\quad \forall t\in [0, T],\label{ndeg_ED_main}
\end{align}
with $\delta$ depending on the functions $U,V$. In particular, it depends only on the parameters specified in the conditions above.    
\end{thm} 
\begin{figure}
\centering
\includegraphics[scale=0.6]{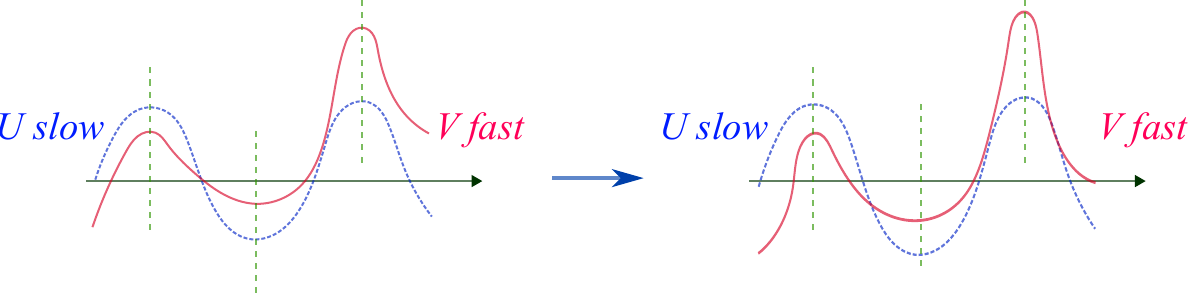}
\caption{Relation between $U,\, V$. The reference $U$ slowly varies, whereas the actual shear $V$ can change fast. However, the two shears share the same critical points. } \label{Figure:1}
\end{figure}
\ifx
\begin{thm}\label{ndeg_main_thm} Consider the solution to the equation \eqref{PS} initiated from the smooth initial data $f_0\in C^\infty(\Torus^2)$. Assume that the shear flow $V(t,y)\in C^3_{t,y}$ satisfies the following structure assumptions on the time interval $[0,T]$: 

a) Phase assumption: There exists a nondegenerate reference profile $U(\cdot,\cdot)\in C_{t,y}^3(\rr_+\times\Torus)$ such that the time-dependent flow $V(t,y)$ and the reference flow $U(t,y)$ share all their nondegenerate critical points $\{y_i(t)\}_{i=1}^N, \, N<\infty$. Moreover,
\begin{align}
\pa_y V(t,y)\pa_y U(t,y)\geq &0,\ \
\|\pa_{yt}U\|_{L_t^\infty([0,T];L^\infty_y)}\leq \nu^{3/4},\quad \|V\|_{L_t^\infty W_y^{3,\infty}}+\|U\|_{L_t^\infty W^{3,\infty}_y}<C,\quad\forall y\in \Torus,\ \forall t\in[0,T].   \label{asmpt0}
\end{align}

b) Shape assumption: 
  there exist $N$ pairwise disjoint open neighborhoods $\{B_{r_i}(y_i(t))\}_{i=1}^N,\, r_i=\mathcal{O}(1)$ and two constants $\mathfrak{C}_0,\ \mathfrak{C}_1> 1$ such that the following estimates hold for $Z(t,y)\in \{V(t,y), U(t,y))\}$
\begin{align}\label{asmp_1}
 \mathfrak{C}_0^{-1}(y-y_i(t))^2\leq& |\pa_yZ|^2\leq \mathfrak{C}_0(y-y_i(t))^2, \quad \mathfrak{C}_0>0, \quad \forall y\in B_{r_i}(y_i(t));\\ 
0<\mathfrak{C}_1^{-1}\leq  |\pa_y Z|\leq& \mathfrak{C}_1 ,\quad \forall y\notin \cup_{i=1}^N B_{r_i}(y_i(t)),\label{asmpt2} 
\end{align}
Then there exists a threshold $\nu_0(U,V)$ such that if $\nu\leq \nu_0$, the following estimate holds\begin{align}
\|f_{k}(t)\|_{L^2}\leq C\|f_{0;k}\|_{L^2}\exp\lf\{-\delta\nu^{1/2}\myr{|k|^{1/2}} t\rg\},\quad \forall t\in [0, T],\label{ndeg_ED_main}
\end{align}
with $C_{ED}$ and $\delta$ depending on the functions $U,V$. In particular, they depend only on the parameters specified in the conditions above.    
\end{thm} 
\fi
\begin{rmk}
We remark that if we consider the solution $V(t,y)=e^{-\nu t}\sin(y)$ to the heat equation $\pa_t V=\nu \pa_{yy}V$ on the torus, the structure conditions are satisfied for time $t\in[0,\mathcal{O}(\nu^{-1+})]$. 
\end{rmk}
\begin{rmk}
In our analysis of the time-dependent shear flows, the dynamics of the critical points are crucial. The main theorem encodes the dynamics of the critical points in the reference shear $U$. The relation between $U, \ V$ is highlighted in Figure \ref{Figure:1}. The condition $\|\pa_{ty}U\|_\infty\leq \nu^{3/4}$ enforces that the critical points of the target shear $V$ cannot move too fast. If this condition is violated, the fluid can trigger mixing and unmixing effects within a short time. Hence, it is not clear whether the enhanced dissipation phenomenon persists. 
\end{rmk}

Finally, we present the following theorem of the Taylor dispersion in an infinite long channel.
\begin{thm}\label{Taylor_thm}
Consider the equation \eqref{PS} in an infinite long channel $\rr\times [-1,1]$ subject to Dirichlet boundary condition $f(t,\pm 1)=0$. The initial data $f_0\in  H^1_0(\rr\times [-1,1])$ is consistent with the boundary condition. Assume that on the time interval $[0,T]$, the shear flow $V\in C_{t}^1C_y^2$ has finitely many critical points $\{y_i(t)\}_{i=1}^{N(t)}$ at every time instance $t\in[0,T]$. Furthermore, there exist four parameters $ \,m_0\in \mathbb{N},\, r_0=\mathcal{O}(1), \, \mathfrak{C}_2,\,\mathfrak{C}_3\geq 1$ such that the following non-degeneracy condition holds around each critical point $y_i(t)$:
\begin{align}\label{Taylorcnstr}
\|\pa_{ty}V\|_\infty\leq \mathfrak{C}_2\nu,\quad |y-y_i|^{m_0}\leq \mathfrak{C}_3|\pa_y V(t,y)|,\quad \forall |y-y_i|\leq r_0,\quad i\in\{1,2,\cdots, N(t)\} .
\end{align}
Then for $0<|k|\leq \nu$, the following estimate holds
\begin{align}
\|f_k(t)\|_{L^2}\leq e\|f_k(0)\|_{L^2}\exp\lf\{-\delta \frac{|k|^2}{\nu} t\rg\},\quad \forall t\in[0,T].
\end{align}
Here, the constant $\delta\in(0,1)$ only depends on the aforementioned properties of the shear flow $V$.
\end{thm}
\begin{rmk}
We remark that the $f_0\in H_0^1$ can be relaxed to $f_0\in L^2. $ This is because the passive scalar equation is smoothing, and the solution will automatically lie in $H_0^1$ for $t>0$. 
\end{rmk}
\begin{rmk}In the paper \cite{CotiZelatiGallay21}, the authors show that one can use the resolvent estimate to derive the enhanced dissipation and Taylor dispersion simultaneously. Here, we observe that one can achieve the same goal utilizing the hypocoercivity techniques. This method is potentially more flexible because it can treat certain time-dependent cases. 
\end{rmk}
The hypocoercivity energy functional introduced in  \cite{BCZ15} is our main tool to prove the main theorems. However, we choose to incorporate time-weights introduced in the papers \cite{WeiZhang19} into our setting. Let us define a parameter and two time weights
\begin{align}
    \ep:=& \nu|k|^{-1},\quad 
    \psi=  \min\{\nu^{1/3}\myr{|k|^{2/3}}t,1\},\quad \phi=\min\{\nu^{1/2}\myr{|k|^{1/2}}t,1\},\quad \zeta=\min\{\nu^{-1}|k|^{2}t, 1\}.\label{ep&t_weights} 
\end{align}
We observe that the derivatives of the time weights are compactly supported:
\begin{align}
    \psi'(t)=\nu^{1/3}\myr{|k|^{2/3}}\mathbbm{1}_{[0,\nu^{-1/3}\myr{|k|^{-2/3}}]}(t),\quad \phi'(t)=\nu^{1/2}\myr{|k|^{1/2}}\mathbbm{1}_{[0,\nu^{-1/2}\myr{|k|^{-1/2}}]}(t),\quad \zeta'(t)=\nu^{-1}\myr{|k|^{2}}\mathbbm{1}_{[0,\nu\myr{|k|^{-2}}]}(t).\\  \label{dt_weights}
\end{align}
To prove Theorem \ref{thm_mon}, Theorem \ref{ndeg_main_thm} and Theorem \ref{Taylor_thm}, we invoke the following hypocoercivity functionals
\begin{align}
\text{Theorem \ref{thm_mon}:\, }&\mathcal{F}[f_k]:=  \|f_k\|_{2}^2+\al \psi\ep^{2/3}\ \|\pa_y f_k\|_{2}^2+\beta \psi^2\ep^{1/3} \ \Re\lan i\text{sign}(k) f_k,\pa_y f_k \ran;\label{hypo_mono}\\
\text{Theorem \ref{ndeg_main_thm}:\, }&\mathcal{G}[f_k]:=  \|f_k\|_{2}^2+\al \phi\ep^{1/2}\ \|\pa_y f_k\|_{2}^2+\beta \phi^2 \ \Re\lan i\text{sign}(k) \pa_y U f_k,\pa_y f_k \ran+\gamma\phi^3\ep^{-1/2} \|\pa_y U f_k\|_2^2;\label{hypo_ndeg}\\
\text{Theorem \ref{Taylor_thm}:\, }&\mathcal{T}[f_k]:=  \|f_k\|_{2}^2+\al \zeta\ \|\pa_y f_k\|_{2}^2+\beta \zeta \ \Re\lan i\text{sign}(k) \pa_y V f_k,\pa_y f_k \ran.&&&&\label{hypo_Tay}
\end{align}
Here, the inner product $\lan \cdot ,\cdot\ran$ is defined in \eqref{inner_prod}.

Through detailed analysis, one can derive the following statements.

\begin{subequations}
\noindent
a) Assume all conditions in Theorem \ref{thm_mon}. There exist parameters $\al=\mathcal{O}(1),\, \beta=\mathcal{O}(1)$ such that the following estimate holds on the time interval $[0,T]$:
\begin{align}
    \mathcal{F}[f_k](t)\leq C\mathcal{F}[f_{0;k}]\exp\lf\{-\delta \nu^{1/3}\myr{|k|^{2/3}}t\rg\}=C\|f_{0;k}\|_2 ^2\exp\lf\{-\delta \nu^{1/3}\myr{|k|^{2/3}}t\rg\},\quad \forall t\in[0,T].\label{Hypo_est_mon}
\end{align}

\noindent
b) Assume all conditions in Theorem \ref{ndeg_main_thm}. Then there exist parameters $\al=\mathcal{O}(1),\, \beta=\mathcal{O}(1),\, \gamma=\mathcal{O}(1)$ such that the following estimate holds for $t\in[0,T]$,
\begin{align}
    \mathcal{G}[f_k](t)\leq C\mathcal{G}[f_{0;k}]\exp\{-\delta \nu^{1/2}\myr{|k|^{1/2}}t\}=C\|f_{0;k}\|_2 ^2\exp\lf\{-\delta \nu^{1/2}\myr{|k|^{1/2}}t\rg\},\quad \forall t\in[0,T].\label{Hypo_est_ndeg}
\end{align}

\noindent
c) Assume all conditions in Theorem \ref{Taylor_thm}. Then there exist parameters $\al=\mathcal{O}(1),\, \beta=\mathcal{O}(1)$ such that the following estimate holds for $t\in[0,T]$,
\begin{align}
    \mathcal{T}[f_k](t)\leq C\mathcal{T}[f_{0;k}]\exp\{-\delta \nu^{-1}\myr{|k|^{2}}t\}=C\|f_{0;k}\|_2^2\exp\lf\{-\delta \nu^{-1}\myr{|k|^{2}}t\rg\},\quad \forall t\in[0,T].\label{Hypo_est_Taylor}
\end{align}
Hence, we can apply the observation that $\|f_k\|_2^2\leq \mathcal{F}[f_k],\, \mathcal{G}[f_k],\,\mathcal{T}[f_k] $ to derive Theorem \ref{thm_mon}, \ref{ndeg_main_thm}, \ref{Taylor_thm}.  
\end{subequations}

We organize the remaining sections as follows: in section \ref{monotone}, we prove Theorem \ref{thm_mon}; in section \ref{nondegenerate}, we prove Theorem \ref{ndeg_main_thm}; in section \ref{Taylor}, we prove Theorem \ref{Taylor_thm}.

\noindent
    {\bf Notations:}
    For two complex-valued functions $f,g$, we define the inner product\begin{align}
         \lan f,g\ran=\int_{D} f\overline{g}dy.\label{inner_prod}
    \end{align}
    Here $D$ is the domain of interest. 
    Furthermore, we introduce the $L^p$-norms ($p\in[1,\infty)$) 
    \begin{align}
        \|f\|_{p}=\|f\|_{L^p}=\left(\int |f|^p dy\right)^{1/p},\quad p\in[1,\infty). 
    \end{align}
    We also recall the standard extension of this definition to the $p=\infty$ case. We further recall the standard definition for Sobolev norms of functions $f(y),\, g(t,y)$:
    \begin{align}
        \|f\|_{W^{m,p}_y}=\lf(\sum_{k=0}^m\|\pa_y^m f\|_{L^p}^p\rg)^{1/p}, \quad p\in [1,\infty];\quad \
        \|g\|_{L_t^q W^{m,p}_y}=\|\|g\|_{W^{m,p}_y}\|_{L_t^q}, \quad p,q\in [1,\infty].
    \end{align}
    We will also use classical notations $H^1=W^{1,2}$ and $H_0^1$ (the $H^1$ functions with zero trace on the boundary). 
    We use the notation $A\approx B$ ($A,B>0$) if there exists a constant $C>0$ such that $\frac{1}{C}B\leq A\leq CB$. Similarly, we use the notation $A\lesssim B$ ($A\gtrsim B$) if there exists a constant $C$ such that $A\leq CB$ ($A\geq B/C$). Throughout the paper, the constant $C$ can depend on the norm $\|V\|_{L_t^\infty W^{3,\infty}_{y}},\, \|U\|_{L_t^\infty W_y^{3,\infty}}$, but it will never depend on $\nu,\, |k|$. The meaning of the notation $C$ can change from line to line.

    \section{Enhanced Dissipation: Strictly Monotone Shear Flows}\label{monotone}In this section, we prove the estimate \eqref{mon_ED_main} for the hypoellitic passive scalar equation $\eqref{k_by_k_eq}_{\sigma=0}$. The proof of the $\sigma=1$ case is similar and simpler. Throughout the remaining part of the paper, we adopt the following notation
\begin{align}\label{notation_simp}
    f(t,y):= \wh f_k(t,y).
\end{align}
Without loss of generality, we assume that 
\begin{align}\label{WLOG}
\pa_y V>0, \quad k\geq 1.\end{align} 

Let us start with a simple observation.
\begin{lem}\label{lem:monotone fnc}
    Assume the relation
    \begin{align}\label{monotone bound req}
    {\al} > \beta^2.
    \end{align}
    Then, the following relations hold
    \begin{align} 
        \frac{1}{2}(\|f\|_2^2 +\al\ep^{2/3}\psi\|\pa_yf\|_2^2) \leq \cF[f] \leq  \frac{3}{2}(\|f\|_2^2 +\al\ep^{2/3}\psi\|\pa_yf\|_2^2),\quad \forall t \in[0,T].\label{equiv_mono}
    \end{align}
\end{lem}
\begin{proof}To prove the estimate, we recall the definition of $\mathcal{F}$ \eqref{hypo_mono}, and estimate it using H\"older inequality,  Young's inequality,
    \begin{align}
        \cF[f] \leq& \|f\|_2^2 + \al\ep^{2/3}\psi\|\pa_yf\|_2^2 + \beta\psi^2\ep^{1/3}\|f\|_2\|\pa_yf\|_2 
        \leq   \lf(1 + \frac{\beta^2}{2\al}\psi^3\rg)\|f\|_2^2 +  \frac{3\al}{2} \ep^{2/3}\psi\|\pa_yf\|_2^2
    \end{align}Similarly, we have the following lower bound,
    \begin{align}
        \cF[f] \geq& \|f\|_2^2 + \al\ep^{2/3}\psi\|\pa_yf\|_2^2 - \beta\ep^{1/3}\psi^2\|f\|_2\|\pa_yf\|_2\geq \lf(1- \frac{\beta^2}{2\al}\psi^3\rg)\|f\|_2^2 + \frac{\al}{2}\ep^{2/3}\psi\|\pa_yf\|_2^2
    \end{align}Since $ {\al} > {\beta^2}$, we obtain that
    \begin{align} 
    \frac{1}{2}\|f\|_2^2 +\frac{1}{2}\al \ep^{2/3}\psi\|\pa_yf\|_2^2\leq \cF[f] \leq \frac{3}{2}\|f\|_2^2 +\frac{3}{2}\al \ep^{2/3}\psi\|\pa_yf\|_2^2,\quad \forall t\in [0,T].
    \end{align}
    This concludes the proof of the lemma.
\end{proof}
By taking the time derivative of the hypocoercivity functional, \eqref{hypo_mono}, we end up with the following decomposition:
\begin{align}\label{T_albe_term}
    \frac{d}{dt}\cF[f]=&\frac{d}{dt} \|f\|_2^2+\al\ep^{2/3}\frac{d}{dt}\lf(\psi\|\pa_{y}f\|_2^2\rg)+\beta\ep^{1/3}\frac{d}{dt}\lf(\psi^2\Re\langle i f,\pa_y f\rangle \rg)
    =:  T_{L^2} + T_\al+ T_\beta.
\end{align}
Through standard energy estimates, we observe that 
\begin{align}
    T_{L^2}=-2\nu\int |\pa_y f|^2 dy- 2\Re i\int  V f \overline{f} dy = - 2\nu\|\pa_yf\|_2^2.\label{T_L2}
\end{align}
The estimates for the $T_\al,\ T_\beta$ terms are trickier, and we collect them in the following technical lemmas whose proofs will be postponed to the end of this section.
\begin{lem}[$\al$-estimate]\label{lem:monotone al}
    For any constant $B>0$, the  following estimate holds on the interval $[0,T]$:
    \begin{align}
       T_\al
       \leq&\al\ep^{2/3}\psi'\|\pa_y f\|_2^2-2\al\psi\ep^{2/3}\nu\|\pa_y^2f\|_2^2+\frac{\beta}{B} \psi^2\ep^{1/3}|k|\|\|\sqrt{|\pa_yV|}f\|_2^2+  \frac{B\al^2}{\beta}\|\pa_y V\|_\infty\nu   \|\pa_yf\|_2^2.\label{monotone al estimate}
    \end{align}
\end{lem}
\begin{lem}[$\beta$-estimate]\label{lem:monotone beta}
    The following estimate holds
    \begin{align}\n
        T_\beta \leq&\frac{\beta}{\sqrt{\al}} \nu^{1/3}\myr{|k|^{2/3}}\mathbbm{1}_{[0,\nu^{-1/3}\myr{|k|^{-2/3}}]}(t) \lf(\|f\|_2^2+ \al\ep^{2/3} \psi\|\pa_yf\|_2^2\rg)\\
        &+\frac{ \beta^2}{\al}\psi^3\nu\|\pa_y f\|_2^2 +{\al}\psi\ep^{2/3}\nu\| \pa_{yy}f\|_2^2-  {\beta}\psi^2\ep^{1/3}|k|\lf\|  \sqrt{|\pa_yV|}|f|\rg\|_{2}^2.\label{monotone beta estimate}
    \end{align}
\end{lem}
We are ready to prove Theorem \ref{thm_mon} with these estimates.
\begin{proof}[Proof of Theorem~\ref{thm_mon}]If $T\leq 2\nu^{-1/3}\myr{|k|^{-2/3}}$, then standard $L^2$-energy estimate yields \eqref{mon_ED_main}. Hence, we assume $T>2\nu^{-1/3}\myr{|k|^{-2/3}}$ without loss of generality. We distinguish between two time intervals, i.e., 
\begin{align}\label{I1_I2}
    \mathcal{I}_1=[0, \nu^{-1/3}\myr{|k|^{-2/3}}], \quad \mathcal{I}_2=[\nu^{-1/3}\myr{|k|^{-2/3}},T].
\end{align} We organize the proof in three steps. 

\noindent
{\bf Step \# 1: Energy bounds.}
Combining the estimates \eqref{T_L2}, \eqref{monotone al estimate}, \eqref{monotone beta estimate}, we obtain that 
    \begin{align}
        \frac{d}{dt}\cF[f]\leq&\al\ep^{2/3}\nu^{1/3}\myr{|k|^{2/3}}\mathbbm{1}_{[0,\nu^{-1/3}\myr{|k|^{-2/3}}]}(t)\|\pa_y f\|_2^2+\frac{\beta}{\sqrt{\al}} \nu^{1/3}\myr{|k|^{2/3}}\mathbbm{1}_{[0,\nu^{-1/3}\myr{|k|^{-2/3}}]}(t) \lf(\|f\|_2^2+ \al\ep^{2/3} \psi\|\pa_yf\|_2^2\rg)\\
& - 2\nu\|\pa_yf\|_2^2-2\al\psi\ep^{2/3}\nu\|\pa_{yy}f\|_2^2+\frac{\beta}{2} \psi^2\ep^{1/3}|k|\lf\|\sqrt{|\pa_yV|}f\rg\|_2^2+  \frac{2\al^2}{\beta} {\nu} \|\pa_yV\|_\infty \|\pa_yf\|_2^2\\
        &+{\al}\psi\ep^{2/3}\nu\| \pa_{yy}f\|_2^2+\frac{ \beta^2}{\al}\psi^3\nu\|\pa_y f\|_2^2 -  {\beta}\psi^2\ep^{1/3}|k|\lf\|  \sqrt{|\pa_yV|}|f|\rg\|_{2}^2\\
        \leq&\frac{\beta}{\sqrt{\al}} \nu^{1/3}\myr{|k|^{2/3}}\mathbbm{1}_{[0,\nu^{-1/3}\myr{|k|^{-2/3}}]}(t) \lf(\|f\|_2^2+ \al\ep^{2/3} \psi\|\pa_yf\|_2^2\rg) \\
        &- \nu\lf(2-\al\mathbbm{1}_{[0,\nu^{-1/3}\myr{|k|^{-2/3}}]}(t) - \frac{2\al^2}{\beta  }\|\pa_yV\|_\infty - \frac{\beta^2}{\al}\psi^3\rg)\|\pa_yf\|_2^2-\frac{ 1}{2}\beta\ep^{1/3}\psi^2|k| \lf\|  \sqrt{|\pa_yV|}f\rg\|_2^2 .
    \end{align}
Now we choose the $\al,\beta $ as follows:
\begin{align}\label{chc_albe_mon}
    \al=\beta=\frac{1}{2(1+\|\pa_y V\|_\infty)}.
\end{align}
Then we check that the condition \eqref{monotone bound req} and the following hold for all $t\in[0,T]$,
    \begin{align}
        2 -\al\mathbbm{1}_{[0,\nu^{-1/3}]}(t) - \frac{2\al^2}{\beta  }\|\pa_yV\|_\infty - \frac{\beta^2}{\al}\psi^3\geq 2 -\frac{1}{2(1+\|\pa_yV\|_\infty)}-\frac{\|\pa_yV\|_\infty}{1+\|\pa_yV\|_\infty}-\frac{1}{2(1+\|\pa_yV\|_\infty)}\geq {1}.
   \end{align}
    As a result, we have \eqref{equiv_mono} and the following,
    \begin{align}
        \frac{d}{dt}\cF[f]\leq&\frac{\beta}{\sqrt{\al}} \nu^{1/3}\myr{|k|^{2/3}}\mathbbm{1}_{[0,\nu^{-1/3}\myr{|k|^{-2/3}}]}(t) \lf(\|f\|_2^2+ \al\ep^{2/3} \psi\|\pa_yf\|_2^2\rg) -{\nu}\|\pa_y f\|_2^2 - \frac{\beta\ep^{1/3}|k|}{2}\psi^2\lf\|\sqrt{|\pa_yV|}f\rg\|_2^2. \\ \label{energy}
    \end{align}

    \noindent 
    {\bf Step \# 2: Initial time layer estimate.}
Thanks to the estimate \eqref{energy} and the equivalence \eqref{equiv_mono}, we have that 
\begin{align}
    \frac{d}{dt}\mathcal{F}[f](t)\leq 2\frac{\beta}{\sqrt{\al}}\nu^{1/3}\myr{|k|^{2/3}}\mathcal{F}[f](t)=\frac{\sqrt{2}}{(1+\|\pa_y V\|_\infty)^{1/2}}\nu^{1/3}\myr{|k|^{2/3}}\mathcal{F}[f](t),\quad \mathcal{F}[f](t=0)=\|f_{0;k}\|_2^2.
\end{align}
By solving this differential inequality, we have that 
\begin{align}\label{Step_2_mono}
    \mathcal{F}[f](t)\leq \exp\lf\{\frac{\sqrt{2}}{(1+\|\pa_y V\|_\infty)^{1/2}}\rg\}\|f_{0;k}\|_2^2,\quad \forall t\in[0, \nu^{-1/3}\myr{|k|^{-2/3}}].
\end{align}

  \noindent 
    {\bf Step \# 3: Long time estimate.}
Now, we focus on the long time interval $\mathcal{I}_2$. On this interval, we have that $\psi\equiv 1$. The estimate \eqref{energy}, together with the lower bound on $|\pa_y V|$ \eqref{V_y}, the choice of $\beta$ \eqref{chc_albe_mon} yields that 
\begin{align}
     \frac{d}{dt}\mathcal{F}[f](t)\leq &-\nu\|\pa_y f\|_2^2-\frac{\nu^{1/3}|k|^{2/3}}{4(1+\|\pa_y V\|_\infty)}\lf\|\sqrt{|\pa_y V|}f\rg\|_2^2\leq -\frac{ \nu^{1/3}|k|^{2/3}}{4(1+\mathfrak c)(1+\|\pa_y V\|_\infty)}(\|f\|_{2}^2+\al\ep^{2/3}\|\pa_y f\|_2^2)\\
     \leq & -\frac{ \nu^{1/3}|k|^{2/3}}{6(1+\mathfrak c)(1+\|\pa_y V\|_\infty)}\mathcal{F}[f](t).
\end{align}
In the last line, we invoked the equivalence \eqref{equiv_mono}.
Hence, for all $t\in[\nu^{-1/3}\myr{|k|^{-2/3}}, T]$ 
\begin{align}
     \mathcal{F}[f](t) \leq  \mathcal{F}[f](t=\nu^{-1/3}\myr{|k|^{-2/3}})\exp\lf\{-\delta\nu^{1/3}\myr{|k|^{2/3}}(t-\nu^{-1/3}\myr{|k|^{-2/3}})\rg\},\quad \delta:=\frac{1}{6(1+\mathfrak c)(1+\|\pa_y V\|_\infty)}.\\ \label{delta_mono}
\end{align}
Thanks to the relation \eqref{Step_2_mono}, we have that 
\begin{align}
    \mathcal{F}[f_k](t) \leq e \|f_{0;k}\|_2^2\exp\lf\{-\delta\nu^{1/3}|k|^{2/3}  t \rg\},\quad \forall t\in[\nu^{-1/3}\myr{|k|^{-2/3}}, T].
\end{align}
This concludes the proof of \eqref{Hypo_est_mon} and Theorem \ref{thm_mon}.
\ifx
    Define $\al$ such that $\beta c^{-1} = 2\frac{\al^2}{\beta c}\|\pa_yV\|_\infty^2$, $\al = \frac{\sqrt{2}\beta}{\|\pa_yV\|_\infty}$.
    \begin{align}
        \frac{d}{dt}\cF[f]\leq& -\ep^{1/3}\psi^2\frac{\beta c}{3}\|f\|_2^2 - \ep\lf(\frac{1}{2} - 3\beta{c^{-1}}\rg)\|\pa_yf\|_2^2
    \end{align}
    Now choose $\beta$ small enough that the following relation holds
    \begin{align}
        \frac{1}{2} - 3\beta{c^{-1}} \geq \frac{\beta c}{3}
    \end{align}
    So that
    \begin{align}
        \frac{d}{dt}\cF[f]\leq& -\frac{\beta c}{3}\ep^{1/3}\lf(\psi^2\|f\|_2^2 + \ep^{2/3}\|\pa_yf\|_2^2\rg)
    \end{align}
    Consider the domain $t\in[0, \ep^{-1/3}]$. The first estimate is shown by noting that at $t=0$, $\cF[f(0)]=\frac{1}{2}\|f(0)\|_2^2$. Furthermore, the derivative is non-positive, so in this domain
    \begin{align}
        \cF[f(t)] \leq \frac{1}{2}\|f(0)\|_2^2
    \end{align}
    Consider the domain $t\geq \ep^{-1/3}$. Then $\min\{\ep^{1/3}t,1\} = 1$ so the derivative estimate reduces to
    \begin{align}
        \frac{d}{dt}\cF[f]\leq& -\frac{\beta c}{3}\ep^{1/3}\lf(\|f\|_2^2 + \ep^{2/3}\|\pa_yf\|_2^2\rg)
    \end{align}
    And using~\ref{equiv_mono}, we have
    \begin{align}
        \frac{d}{dt}\cF[f]\leq& -C\ep^{1/3}\lf(\cF[f]\rg)
    \end{align}
    The inequality solution to which is
    \begin{align}
        \cF[f(t)] \leq& \cF[f(\ep^{-1/3})]e^{-C\ep^{1/3}\lf(t-\ep^{-1/3}\rg)}\\
        =& \cF[f(\ep^{-1/3})]\lf(e^C\rg)e^{-C\ep^{1/3}t}
    \end{align}
    Now the two domain inequalities can be combined at $t=\ep^{-1/3}$ to form an inequality which is valid for $t\geq\ep^{-1/3}$.
    \begin{align}
        \cF[f(t)] \leq& \frac{1}{2}\lf(e^{C}\rg)\|f(0)\|_2^2e^{-C\ep^{1/3}t}
    \end{align}
    Therefore there is some larger constant $C(\al, \beta, \|\pa_yV\|_\infty^2)$ so that the inequality holds for $t\in[0, \ep^{-1/3}]$.
    \begin{align}
        \cF[f(t)] \leq& C\|f(0)\|_2^2e^{-\delta\ep^{1/3}t}
    \end{align}
    where $\delta\sim\beta$ is small.\fi
\end{proof}

Finally, we collect the proofs of the technical lemmas.
\begin{proof}[Proof of Lemma~\ref{lem:monotone al}]
    We recall the definition of $T_\al$ \eqref{T_albe_term}. Invoking the equation  \eqref{k_by_k_eq} and integration by parts yields that 
    \begin{align}
        T_\al=& \al\psi' \ep^{2/3}\|\pa_yf\|_2^2 + \al\psi\ep^{2/3}\frac{d}{dt}\|\pa_y f\|_2^2 
         =  \al\psi' \ep^{2/3}\|\pa_yf\|_2^2 + 2\al\psi\ep^{2/3}\Re\int \pa_y(\nu\pa_{y}^2f - ikVf) \overline{\pa_y f}dy\\
         =& \al\psi' \ep^{2/3}\|\pa_yf\|_2^2 + 2\al\psi\ep^{2/3}\Re\int\lf(\nu \pa_{y}^3f- ik\pa_yVf - ikV\pa_yf\rg)\overline{\pa_y f}dy\\
         =&\al\psi' \ep^{2/3}\|\pa_yf\|_2^2 + 2\al\psi\ep^{2/3}\lf(\nu\Re\int\pa_{y}^3f\overline{\pa_yf}dy - \Re\int ik\pa_yVf\overline{\pa_yf}dy - \Re\int ikV\pa_yf\overline{\pa_yf}dy\rg)\\
         =& \al\psi' \ep^{2/3}\|\pa_yf\|_2^2 - 2\al\psi\ep^{2/3}\lf( \nu\|\pa_{y}^2f\|_2^2 + \Re\int ik\pa_yVf\overline{\pa_yf}dy\rg)\\
         \leq& \al\psi' \ep^{2/3}\|\pa_yf\|_2^2 - 2\al\psi\ep^{2/3}\nu\|\pa_{y}^2f\|_2^2 + 2\al\psi\ep^{2/3}|k|\|\pa_yV\|_{\infty}^{1/2}\|\sqrt{|\pa_yV|} f\|_2\|\pa_yf\|_2.
    \end{align}
    An application of Young's inequality yields \eqref{monotone al estimate}.
\end{proof}
\begin{proof}[Proof of Lemma~\ref{lem:monotone beta}] The estimate of the $T_\beta$ term in \eqref{T_albe_term} is technical. Hence, we further decompose it into three terms:
	\begin{align}
	T_\beta =& 2\beta \psi \psi'\ep^{1/3}\Re\langle i f,\pa_y f\rangle + \beta\psi^2\ep^{1/3} \Re\int i\pa_tf\overline{\pa_yf}dy +  \beta\psi^2\ep^{1/3}\Re\int if\overline{\pa_{yt}f}dy=:T_{\beta;1}+T_{\beta;2}+T_{\beta;3}. \label{T_beta123}
 \end{align}
 We estimate these terms one by one. To begin with, we have the following bound for the $T_{\beta;1}$:
 \begin{align}
|T_{\beta;1}|\leq& 2\frac{\beta}{\sqrt{\al}} \nu^{1/3}\myr{|k|^{2/3}}\mathbbm{1}_{[0,\nu^{-1/3}\myr{|k|^{-2/3}}]}(t)\sqrt{\psi} \|f\|_2 (\sqrt{\al}\ep^{1/3} \sqrt{\psi}\|\pa_yf\|_2) \\
\leq& \frac{\beta}{\sqrt{\al}} \nu^{1/3}\myr{|k|^{2/3}}\mathbbm{1}_{[0,\nu^{-1/3}\myr{|k|^{-2/3}}]}(t) \lf(\|f\|_2^2+ \al\ep^{2/3} \psi\|\pa_yf\|_2^2\rg).   \label{T_beta1}
 \end{align}
 Next we compute the term $T_{\beta;2}$ using the equation \eqref{k_by_k_eq}and the assumption $\pa_y V>0$ \eqref{WLOG}:
 \begin{align}
	T_{\beta;2}=&  \beta\psi^2\ep^{1/3}\Re\int i\lf(\nu\pa_{yy}f - iVf\rg)\overline{\pa_yf}dy 
		=   \beta\psi^2\ep^{1/3}\lf(\nu\Re\int i\pa_{yy}f\overline{\pa_yf}dy + k\Re\int V \pa_y \lf(\frac{|f|^2}{2}\rg)dy \rg)\\
		=& \beta\psi^2\ep^{1/3}\lf(\nu\Re\int i\pa_{yy}f\overline{\pa_yf}dy - \frac{k}{2}\Re\int  |f|^2\pa_yV dy\rg). \label{T_beta_2}
	\end{align}
Finally, we focus on the $T_{\beta;3}$ term in \eqref{T_beta123}. Recalling that $0\leq \pa_yV\in \mathbb R$, we have that 
  \begin{align}
	T_{\beta;3}=&   \beta\psi^2\ep^{1/3}  \Re\int if\overline{\lf(\nu\pa_{y}^3f- ik\pa_yVf - ikV\pa_yf\rg)}dy \\
		=& \beta\psi^2\ep^{1/3}\lf(-\nu\Re\int i\pa_y f\overline{\pa_{y}^2f}dy- k\Re\int f\overline{\pa_yVf}dy - k\Re\int V \pa_y \lf(\frac{|f|^2}{2}\rg)dy\rg)\\
  =& -\beta\psi^2\ep^{1/3}\nu\Re\int i\pa_y f\overline{\pa_{y}^2f}dy- \frac{\beta}{2}\psi^2\ep^{1/3}k\Re\int |f|^2 \pa_yV dy.\label{T_beta3}
\end{align}
Combining the estimates \eqref{T_beta1}, \eqref{T_beta_2}, \eqref{T_beta3}, we have that 
\begin{align}
    T_\beta\leq& \frac{\beta}{\sqrt{\al}} \nu^{1/3}\myr{|k|^{2/3}}\mathbbm{1}_{[0,\nu^{-1/3}\myr{|k|^{-2/3}}]}(t) \lf(\|f\|_2^2+ \al\ep^{2/3} \psi\|\pa_yf\|_2^2\rg)-2\beta\psi^2\ep^{1/3}\nu\Re\int i\pa_y f\overline{\pa_{y}^2f}dy\\
    &-  {\beta}\psi^2\ep^{1/3}|k|\lf\|  \sqrt{|\pa_yV|}|f|\rg\|_{2}^2\\
    \leq & \frac{\beta}{\sqrt{\al}} \nu^{1/3}\myr{|k|^{2/3}}\mathbbm{1}_{[0,\nu^{-1/3}\myr{|k|^{-2/3}}]}(t) \lf(\|f\|_2^2+ \al\ep^{2/3} \psi\|\pa_yf\|_2^2\rg)+\frac{\beta^2}{\al}\psi^3\nu\|\pa_y f\|_2^2 +{\al}{}\psi\ep^{2/3}\nu\| \pa_{yy}f\|_2^2\\
    &-  {\beta}\psi^2\ep^{1/3}|k|\lf\|  \sqrt{|\pa_yV|}|f|\rg\|_{2}^2.
\end{align}
\ifx
  \begin{align}
		=& 2\beta\psi\ep^{2/3}\|f\|_2\|\pa_yf\|_2\\
        +& \beta\psi^2\ep^{1/3}\lf(\ep\Re\int i\pa_{yy}f\overline{\pa_yf}dy - 2\ep\Re\int if\overline{\pa_yf}dy + \ep\Re\int if\overline{\pa_{yyy}f}dy - \Re\int f\overline{\pa_yVf}dy\rg)\\
		=& 2\beta\psi\ep^{2/3}\|f\|_2\|\pa_yf\|_2\\
        +& \beta\psi^2\ep^{1/3}\lf(\ep\Re\int i\pa_{yy}f\overline{\pa_yf}dy - 2\ep\Re\int if\overline{\pa_yf}dy - \ep\Re\int i\pa_yf\overline{\pa_{yy}f}dy - \Re\int f\overline{\pa_yVf}dy\rg)\\
		=& 2\beta\psi\ep^{2/3}\|f\|_2\|\pa_yf\|_2 + \beta\psi^2\ep^{1/3}\lf(-2\ep\Re\int if\overline{\pa_yf}dy - \Re\int f\overline{V\pa_yf}dy\rg)\\
		\leq& 2\beta\psi\ep^{2/3}\|f\|_2\|\pa_yf\|_2 + 2\beta\psi^2\ep^{4/3}\|f\|_2\|\pa_yf\|_2 - \psi^2\|\sqrt{\pa_yV}f\|_2^2\\
        \leq& \frac{\beta c}{4}\psi^2\ep^{1/3}\|f\|_2^2 + 2\beta\ep{c^{-1}}\|\pa_yf\|_2^2 + 2\beta\psi^2\ep^{4/3}\|f\|_2\|\pa_yf\|_2 - \psi^2\|\sqrt{\pa_yV}f\|_2^2\\
		\leq& \frac{\beta c}{4}\beta\psi^2\ep^{1/3}\|f\|_2^2 + 2\beta\ep{c^{-1}}\|\pa_yf\|_2^2\\
        +& \beta\psi^2\ep^{4/3}\|f\|_2^2 + \beta\psi^2\ep^{4/3}\|\pa_yf\|_2^2 - \beta\psi^2\ep^{1/3}\|\sqrt{\pa_yV}f\|_2^2\\
        \leq& \frac{\beta c}{4}\beta\psi^2\ep^{1/3}\|f\|_2^2 + 2\beta\ep{c^{-1}}\|\pa_yf\|_2^2\\
        +& \beta\psi^2\ep^{4/3}\|f\|_2^2 + \beta\ep^{4/3}\|\pa_yf\|_2^2 - \beta\psi^2\ep^{1/3}\|\sqrt{\pa_yV}f\|_2^2\\
		\leq& \frac{\beta c}{4}\psi^2\ep^{1/3}\|f\|_2^2 + 2\beta\ep{c^{-1}}\|\pa_yf\|_2^2\\
        +& \beta\psi^2\ep^{4/3}\|f\|_2^2 + \beta\ep^{4/3}\|\pa_yf\|_2^2 - \beta\psi^2\ep^{1/3}\min_{y,t}(\sqrt{\pa_yV})^2\|f\|_2^2\\
        \leq& \frac{\beta c}{4}\psi^2\ep^{1/3}\|f\|_2^2 + 2\beta\ep{c^{-1}}\|\pa_yf\|_2^2\\
        +& \beta\psi^2\ep^{4/3}\|f\|_2^2 + \beta\ep^{4/3}\|\pa_yf\|_2^2 - \beta c\psi^2\ep^{1/3}\|f\|_2^2\\
        \leq& -\ep^{1/3}\psi^2\lf(\frac{3\beta c}{4} - \beta\ep\rg)\|f\|_2^2 + \ep\lf(2\beta{c^{-1}} + \beta\ep^{1/3}\rg)\|\pa_yf\|_2^2
	\end{align}
 \fi
\end{proof}

\ifx
We consider the following passive scalar equation in the channel $\T\times\R$ or $\T\times[0,1]$
\begin{align}
    \pa_t f_{\neq}+ V(t,y) \pa_x f_{\neq}=\ep\Delta f_{\neq},\quad f_{\neq}(t=0)=f_{\text{in};\neq}.
\end{align}
Assume that we are on $\Torus \times \rr$ and $f_\nq$ decays to zero fast enough near $y=\infty$. The time-dependent shear flow $V$ is strictly monotone.
With monotone flow, we will assume that 
\begin{align}\label{V_y assumption}
    \min_{t\in \rr, y\in \rr}|\pa_yV|\geq c>0.
\end{align}
We take the Fourier transform in the $x$-variable to obtain the following
\begin{align*}
    \pa_t \wh f_k+V(t,y)ik \wh f_k=\ep \pa_{yy}\wh f_k-\ep|k|^2 \wh f_k.
\end{align*}
\myr{It turns out that we can focus on the $k=1$ case:
\begin{align}
    \pa_t \wh f_1+V(t,y)i \wh f_1=\ep \pa_{yy}\wh f_1-\ep \wh f_1.
\end{align}}
Define the hypocoercivity functional as
\begin{align}
    \label{monotone fnc}
    \cF[f]=\frac{1}{2}\|f\|_2^2+\al\psi\lf(\frac{\ep}{|k|}\rg)^{2/3}\|\pa_{y}f\|_2^2+\beta\psi^2\lf(\frac{\ep}{|k|}\rg)^{1/3}\Re\langle i k f,\pa_y f\rangle.
\end{align}
Here $\al, \beta$ are three parameters to be chosen. The following equivalence relation holds.

\begin{thm}
    With the previous lemmas, there are $\al$ and $\beta$ such that:
    For $t\in[0, \ep^{-1/3}]$, 
    \begin{align}
        \cF[f(t)] \lesssim \|f(0)\|_2^2
    \end{align}
    For $t>\ep^{-1/3}$,
    \begin{align}
        \cF[f(t)] \lesssim \cF[f(\ep^{-1/3})]e^{-\delta\ep^{1/3}\lf(t-\ep^{-1/3}\rg)}
    \end{align}
    Moreover, the following estimate holds on the whole time interval.
    \begin{align}
        \cF[f(t)] \leq C\|f(0)\|_2^2e^{-\delta\ep^{1/3}t}
    \end{align}
\end{thm}\fi
    \section{Enhanced Dissipation: Nondegenerate Shear Flows}\label{nondegenerate}
     In this section, we prove the estimate~\eqref{ndeg_ED_main} for the hypoellitic passive scalar equation $\eqref{k_by_k_eq}_{\sigma=0}$. 
Without loss of generality, we assume that $k \geq 1.$ Let us start with a lemma. 
\begin{lem}
    Consider the flow $V(t,y)$ and the reference flow $U(t,y)$ as in Theorem \ref{ndeg_main_thm}.  There exists a constant $C_\ast(\mathfrak {C}_0,\mathfrak {C}_1)> 1$ such that the following estimate holds
    \begin{align}\label{equiv_UV}
        C_\ast^{-1}|\pa_yU(t,y)|\leq |\pa_y V(t,y)|\leq C_\ast  |\pa_yU(t,y)|,\quad \forall  y\in \Torus,\quad\forall t\in[0,T]. 
    \end{align}
    \end{lem}
\begin{proof}
We distinguish between two cases: a) $y\in B_{r}(y_i(t))$; b) $y\in (\cup_{i=1}^N B_{r}(y_i(t)))^c$. If $y\in B_{r}(y_i(t))$, by \eqref{asmp_1},
    \begin{align} 
         |\pa_y V(t,y)|\leq \mathfrak{C}_0^{1/2}|y-y_i(t)|\leq   \mathfrak{C}_0|\pa_yU(t,y)|,\quad |\pa_yU(t,y)|  \leq \mathfrak{C}_0^{1/2}|y-y_i(t)|\leq \mathfrak{C}_0 |\pa_y V(t,y)| . 
    \end{align}
In case b), since $|\pa_y V|,|\pa_yU|\in[\mathfrak{C}_1^{-1},\mathfrak{C}_1]$, the relation \eqref{equiv_UV} is direct.
\end{proof}
\begin{lem}\label{lem:nondegenerate fnc}
    Assume the relation
    \begin{align}\label{ndeg_bnd_req}
	\myr{\beta^2 \leq \al\gamma.}
    \end{align}
    Then, the following equivalence relation concerning the functional $\mathcal{G}$ \eqref{hypo_ndeg} holds
    \begin{align}
        \|f\|_2^2 + \frac{1}{2}\lf(\al\phi\ep^{1/2}\|\pa_yf\|_2^2 + \gamma\phi^3\ep^{-1/2}\|\pa_yUf\|_2^2\rg) \leq \mathcal{G}[f] \leq \|f\|_2^2 + \frac{3}{2}\lf(\al\phi\ep^{1/2}\|\pa_yf\|_2^2 + \gamma\phi^3\ep^{-1/2}\|\pa_yUf\|_2^2\rg).\\\label{equiv_ndg}
    \end{align}
\end{lem}
\begin{proof}
    We recall the definition of $\mathcal{G}$ \eqref{hypo_ndeg}, and estimate $\mathcal{G}[f]$ using H\"older inequality and Young's inequality,
    \begin{align}
        \mathcal{G}[f] \leq& \|f\|_2^2 + \al\phi\ep^{1/2}\|\pa_yf\|_2^2 + \beta\phi^2\|\pa_yUf\|_2^2\|\pa_yf\|_2^2 + \gamma\phi^3\ep^{-1/2}\|\pa_yUf\|_2^2\\
        \leq& \|f\|_2^2 + \frac{3\al}{2}\phi\ep^{1/2}\|\pa_yf\|_2^2 + \lf(\gamma + \frac{\beta^2}{2\al}\rg)\phi^3\ep^{-1/2}\|\pa_yUf\|_2^2.
    \end{align}
    Similarly, we have the lower bound,
    \begin{align}
        \mathcal{G}[f] \geq \|f\|_2^2 + \frac{\al}{2}\phi\ep^{1/2}\|\pa_yf\|_2^2 + \lf(\gamma - \frac{\beta^2}{2\al}\rg)\phi^3\ep^{-1/2}\|\pa_yUf\|_2^2.
    \end{align}
    Since \eqref{ndeg_bnd_req} implies that $\frac{\beta^2}{2\al} \leq \frac{\gamma}{2}$, we obtain that
    \begin{align}
        \|f\|_2^2 + \frac{1}{2}\al\phi\ep^{1/2}\|\pa_yf\|_2^2 + \frac{1}{2}\gamma\phi^3\ep^{-1/2}\|\pa_yUf\|_2^2 \leq \mathcal{G}[f(t)] \leq \|f\|_2^2 + \frac{3}{2}\al\phi\ep^{1/2}\|\pa_yf\|_2^2 + \frac{3}{2}\gamma\phi^3\ep^{-1/2}\|\pa_yUf\|_2^2.
    \end{align}
    This concludes the proof of the lemma.
\end{proof}
By taking the time derivative of the hypocoercivity functional, \eqref{hypo_mono}, we end up with the following decomposition:
\begin{align}
    \frac{d}{dt}\mathcal{G}[f(t)] =& \frac{d}{dt}\|f\|_2^2 + \al \ep^{1/2}\frac{d}{dt}\lf(\phi\|\pa_{y}f\|_2^2\rg) + \beta\frac{d}{dt}\lf(\phi^2\Re\langle i\pa_yUf,\pa_yf\rangle\rg) + \gamma\ep^{-1/2}\frac{d}{dt}\lf(\phi^3\|\pa_yUf\|_2^2\rg)\\
    =:& \mathbb T_{L^2} + \mathbb T_\al+ \mathbb T_\beta + \mathbb T_\gamma.\label{ndeg_T_albe_term}
\end{align}
The estimates for the $\mathbb T_\al,\ \mathbb T_\beta$, and $\mathbb T_\gamma$ terms are tricky, and we collect them in the following technical lemmas whose proofs will be postponed to the end of this section.
\begin{lem}[$\al$-estimate]\label{lem:nondegenerate al}
	The following estimate holds on the interval $[0,T]$:
	\begin{align}\label{deg_al_est}
            \mathbb{T}_\al \leq \al\nu\lf(1+\frac{4\al}{\beta}C_\ast^3\rg)\|\pa_yf\|_2^2 - 2\al\phi\ep^{1/2}\nu\|\pa_{y}^2f\|_2^2 +  \frac{\beta\phi^2}{4C_\ast} |k|\|\pa_yUf\|_2^2.
            \end{align}
Here, the constant $C_\ast$ is defined in \eqref{equiv_UV}.  
\end{lem}
\begin{lem}[$\beta$-estimate]\label{lem:nondegenerate beta}
	The following estimate holds
	\begin{align}
            \mathbb{T}_\beta \leq& \lf(\frac{1}{4}+ 4\beta C_\ast \rg)\nu\|\pa_yf\|_2^2+2{\al} \phi\ep^{1/2}\nu\|\pa_y^2f\|_2^2  - \frac{3}{4}\frac{\beta\phi^2|k|}{C_\ast}\|\pa_{y}Uf\|_2^2\\
        &+\lf(\frac{\beta\phi^2}{|k|^{1/2}}+\frac{\beta\phi}{2\al}\|\pa_{yy}U\|_\infty^2\rg)\beta \phi^2\nu^{1/2}|k|^{1/2}\|f\|_2^2  + \lf(\frac{3\beta^2}{4\al\gamma}\rg)\gamma\phi^3\ep^{-1/2}\nu\|\pa_{y}U\pa_yf\|_2^2 . \label{ndeg_beta_est}
	\end{align}
Here the constant $C_\ast$ is defined in \eqref{equiv_UV}.
\end{lem}
\begin{lem}[$\gamma$-estimate]\label{lem:nondegenerate gamma}
	The following estimate holds on the interval $[0,T]$
	\begin{align}
            \mathbb{T}_\gamma \leq& \lf(\frac{3\gamma C_\ast}{\beta}+\frac{1}{4}\rg)\frac{\beta|k|\phi^2\|\pa_yUf\|_2^2}{C_\ast} +\lf( \frac{4C_\ast\gamma^2\phi^2}{\beta^2|k|^{1/2}}+\frac{4\gamma}{\beta}\phi\|\pa_{yy}U\|_\infty^2\rg)\beta\phi^2\nu^{1/2}|k|^{1/2}\|f\|_2^2  -\gamma\phi^3\ep^{-1/2}\nu\|\pa_yU\pa_yf\|_2^2.\label{ndeg_gamma_est}
	\end{align}
Here the $C_\ast$ is defined in \eqref{equiv_UV}. 
\end{lem}
These estimates allow us to prove Theorem \ref{ndeg_main_thm}.
\begin{proof}[Proof of Theorem~\ref{ndeg_main_thm}]
    If $T\leq 2\nu^{-1/2}\myr{|k|^{-1/2}}$, then standard $L^2$-energy estimate yields \eqref{ndeg_ED_main}. Hence, we assume $T>2\nu^{-1/2}\myr{|k|^{-1/2}}$ without loss of generality. We distinguish between two time intervals, i.e., 
    \begin{align}\label{ndeg_I1_I2}
        \mathcal{I}_1=[0, \nu^{-1/2}\myr{|k|^{-1/2}}], \quad \mathcal{I}_2=[\nu^{-1/2}\myr{|k|^{-1/2}},T].
    \end{align}
    We organize the proof into three steps. In step \# 1, we choose the $\al, \beta,\gamma$ parameters and derive the energy dissipation relation. In step \# 2, we estimate the functional $\cG$ in the time interval $\mathcal{I}_1$. In step \# 3, we estimate the functional $ \cG$ in the time interval $\mathcal{I}_2$ and conclude the proof. 
    
    \noindent {\bf Step \# 1: Energy bounds.}        Combining the estimates \eqref{T_L2}, \eqref{deg_al_est}, \eqref{ndeg_beta_est}, \eqref{ndeg_gamma_est}, we obtain that
    \begin{align}
        \frac{d}{dt}\mathcal{G}[f(t)] 
        \leq&  - \lf(\frac{7}{4}-  \myr{\al} - \frac{4\al^2}{\beta}C_\ast^3- 4\beta C_\ast \rg)\nu\|\pa_yf\|_2^2- \lf(\frac{1}{4} - \frac{3\gamma C_\ast}{\beta}\rg)\frac{\beta|k|\phi^2}{C_\ast}\|\pa_{y}Uf\|_2^2 \\
        &+\lf(\frac{\beta \phi^2}{|k|^{1/2}} + \frac{\beta\phi}{2\al}\|\pa_{yy}U\|_\infty^2+ \frac{4C_\ast\gamma^2\phi^2}{\beta^2|k|^{1/2}} + \frac{4\gamma}{\beta}\phi\|\pa_{yy}U\|_\infty^2\rg)\beta\phi^2\nu^{1/2}|k|^{1/2}\|f\|_2^2\\
        &- \gamma\lf(1 - \frac{3\beta^2}{4\al\gamma}\rg)\phi^3\nu\ep^{-1/2}\|\pa_{y}U\pa_yf\|_2^2.
    \end{align}
    We choose $\al$, $\gamma$ in terms of $\beta(\leq 1)$ as follows 
    \begin{align}\label{chc_al_ga}
        \al = \frac{\beta^{1/2}}{4\myr{C_\ast^{3/2}}}, \quad\quad \gamma = 4\beta^{3/2}\myr{C_\ast^{3/2}}.
    \end{align}
    The resulting differential inequality is
    \begin{align*}
    \frac{d}{dt}\cG[f(t)]\leq &  - \lf(\frac{5}{4} - 4\beta C_\ast \rg)\underbrace{\ep|k|}_{=\nu}\|\pa_yf\|_2^2        - \lf(\frac{1}{4} - 12\beta^{1/2}C_\ast^{5/2} \rg)\frac{\beta|k|\phi^2}{C_\ast}\|\pa_{y}Uf\|_2^2 \\
        &+\underbrace{\lf(\beta + 2\beta^{1/2}C_\ast^{3/2} +  {64C_\ast^4\beta}  + 16 {\beta^{1/2}C_\ast^{3/2}}  \rg)}_{\leq 83\beta^{1/2}C_\ast^4}\max\{1,\|\pa_{yy}U\|_\infty^2\}\beta\phi^2\underbrace{\ep^{1/2}|k|}_{=\nu^{1/2}|k|^{1/2}}\|f\|_2^2\\
        &- \frac{\gamma}{4}\phi^3\nu\ep^{-1/2}\|\pa_{y}U\pa_yf\|_2^2.
    \end{align*}
    Now we invoke the spectral inequality \eqref{spectral} to obtain that 
      \begin{align*}
    \frac{d}{dt}\cG[f(t)]\leq &  - \lf(\frac{5}{4} - 4\beta C_\ast-83\beta^{3/2}C_\ast^4\max\{1,\|\pa_{yy}U\|_\infty^2\}\rg)\nu\|\pa_yf\|_2^2   \\
        &     - \lf(\frac{1}{4C_\ast} - 12\beta^{1/2}C_\ast^{3/2}-83\beta^{1/2}C_\ast^4\mathfrak{C}_{\text{spec}} \max\{1,\|\pa_{yy}U\|_\infty^2\}\rg)\beta|k|\phi^2\|\pa_{y}Uf\|_2^2 .
    \end{align*}
    Hence we can choose
    \begin{align}
    \beta=\beta(C_\ast,\mathfrak{C}_{\mathrm{spec}},\|\pa_{yy}U\|_\infty)<1\label{chc_beta}
    \end{align}
    small enough, invoke the spectral inequality \eqref{spectral} and the equivalence relation \eqref{equiv_ndg}  to obtain that 
    \begin{align*}    \frac{d}{dt}\mathcal{G}[f(t)]\leq&- \frac{1}{2}\ep|k|\|\pa_y f\|_2^2-\frac{\beta}{8C_\ast}|k|\phi^2\|\pa_{y}Uf\|_2^2\\
    \leq&-\frac{\beta\phi^2}{16\mathfrak{C}_{\mathrm{spec}}C_\ast}\ep^{1/2}|k|\|f\|_2^2-\frac{1}{4}\nu^{1/2}|k|^{1/2}\ep^{1/2}\phi\|\pa_y f\|_2^2-\frac{\beta}{16C_\ast}\nu^{1/2}|k|^{1/2}\ep^{-1/2}\phi^3\|\pa_{y}Uf\|_2^2\\
    \leq& -\delta(\beta,\mathfrak{C}_{\text{spec}}^{-1}, C_\ast^{-1})\nu^{1/2}|k|^{1/2}\mathcal{G}[f].\label{dfn_del}
    \end{align*}   Finally, we observe that the parameter $\delta$ depends only on three parameters $C_{\ast},\,\mathfrak{C}_{\text{spec}}$ and $\|\pa_{yy}U\|_\infty$.
    
    \noindent
    {\bf Step \# 2: Initial time layer estimate.}
    This step is similar to the argument in the strictly monotone shear case. Thanks to the energy dissipation relation \eqref{dfn_del}, we obtain that
    \begin{align}
        \cG[f_k](t)\leq C\|f_{0;k}\|_{L^2}^2,\quad\forall t\in[0, \nu^{-1/2}|k|^{-1/2}].
    \end{align}
    
    \noindent
    {\bf Step \# 3: Long time estimate.}
    Assume $t\geq \nu^{-1/2}|k|^{-1/2}$. Thanks to the energy dissipation relation \eqref{dfn_del}, we obtain
    \begin{align}
        \frac{d}{dt}\cG[f]\leq& -\delta\nu^{1/2}|k|^{1/2}\cG[f].
    \end{align}Hence, we obtain that 
    \begin{align}
        \cG[f(t)] \leq& \cG[f(\nu^{-1/2}|k|^{-1/2})]e^{-\delta\nu^{1/2}\myr{|k|^{1/2}}\lf(t-\nu^{-1/2}\myr{|k|^{-1/2}}\rg)}
        \leq e\cG[f(0)] e^{-\delta\nu^{1/2}\myr{|k|^{1/2}}t}= e\|f(0)\|_2^2 e^{-\delta\nu^{1/2}\myr{|k|^{1/2}}t}. 
    \end{align}
    Now, the results from step 2 and 3 yields \eqref{Hypo_est_ndeg}.
\end{proof}
We conclude the section by providing the details of the proof of Lemma \ref{lem:nondegenerate al}, \ref{lem:nondegenerate beta}, and \ref{lem:nondegenerate gamma}. 
\begin{proof}[Proof of Lemma~\ref{lem:nondegenerate al}] We recall the definition of $\mathbb{T}_\al$~\eqref{ndeg_T_albe_term}. Invoking the equation~\eqref{k_by_k_eq} and integration by parts yields that
    \begin{align}
         \mathbb{T}_\al=& \al\phi'\ep^{1/2}\|\pa_yf\|_2^2 + \al\phi\ep^{1/2}\frac{d}{dt}\|\pa_yf\|_2^2 
         =  \al\phi'\ep^{1/2}\|\pa_yf\|_2^2 + 2\al\phi\ep^{1/2}\Re\int \pa_y(\nu\pa_{y}^2f - ikVf) \overline{\pa_y f}dy\\
         =& \al\phi'\ep^{1/2}\|\pa_yf\|_2^2 - 2\al\phi\ep^{1/2}\lf( \nu\|\pa_{y}^2f\|_2^2 + \Re\int ik\pa_yVf\overline{\pa_yf}dy\rg).
    \end{align}
         Now we apply H\"older inequality, the expression \eqref{dt_weights}, and the equivalence relation \eqref{equiv_UV} to obtain that 
    \begin{align}
         \mathbb{T}_\al\leq& \al\nu\|\pa_yf\|_2^2 - 2\al\phi\ep^{1/2}\nu\|\pa_{y}^2f\|_2^2 + 2\al\phi\ep^{1/2}|k|\|\pa_yVf\|_{2}\|\pa_yf\|_2\\
         \leq& \al\nu\|\pa_yf\|_2^2 - 2\al\phi\ep^{1/2}\nu\|\pa_{y}^2f\|_2^2 + \frac{4\al^2}{\beta}C_\ast^3\nu\|\pa_yf\|_2^2+\frac{\beta\phi^2|k|}{4C_\ast}\|\pa_{y}Uf\|_2^2.
    \end{align}
   This is \eqref{deg_al_est}.
\end{proof}
\begin{proof}[Proof of Lemma~\ref{lem:nondegenerate beta}]
    The estimate of the $\mathbb{T}_\beta$ term in \eqref{T_albe_term} is technical. We further decompose it into four terms and estimate them one by one:
    \begin{align}
        \mathbb{T}_\beta =& 2\beta\phi\phi'\Re\langle i\pa_{y}Uf,\pa_yf\rangle + \beta\phi^2\Re\int i\pa_{ty}Uf\overline{\pa_yf}dy + \beta\phi^2\Re\int i\pa_{y}U\pa_tf\overline{\pa_yf}dy + \beta\phi^2\Re\int i\pa_{y}Uf\overline{\pa_{yt}f}dy\\
        =:& \mathbb{T}_{\beta;1} + \mathbb{T}_{\beta;2} + \mathbb{T}_{\beta;3} + \mathbb{T}_{\beta;4}.\label{T_beta1234}
    \end{align}
    To begin with, we apply the expression \eqref{dt_weights}, the H\"older and Young's inequalities to derive the following bound for the $\mathbb{T}_{\beta;1}$ term, 
    \begin{align}
        \mathbb{T}_{\beta;1} \leq& 2\beta\phi\nu^{1/2}|k|^{1/2}\|\pa_{y}Uf\|_2\|\pa_yf\|_2 \leq \frac{\beta\phi^2|k|}{4C_\ast} \|\pa_{y}Uf\|_2^2 +  4\beta C_\ast \nu  \|\pa_yf\|_2^2.
    \end{align}
    Next we estimate the term $\mathbb{T}_{\beta;2}$ using the assumption~\eqref{asmpt0},
    \begin{align}
        \mathbb{T}_{\beta;2} \leq \beta\phi^2\|\pa_{ty}U\|_\infty\|f\|_2\|\pa_yf\|_2 
        \leq \beta^2\phi^4\nu^{1/2}\|f\|_2^2 + \frac{1}{4}\nu\|\pa_yf\|_2^2.
    \end{align}
    To estimate the $\mathbb{T}_{\beta;3}$-term in \eqref{T_beta1234}, we apply integration by parts and obtain that 
    \begin{align}
        \mathbb{T}_{\beta;3} =& \beta\phi^2\Re\int i\pa_{y}U\lf(\nu\pa_{yy}f-iVkf\rg)\overline{\pa_yf}dy = \beta\phi^2\lf(\nu\Re\int i\pa_{y}U\pa_{yy}f\overline{\pa_yf}dy + k\Re\int \pa_{y}UVf\overline{\pa_yf}dy\rg)\\
        \leq& \beta\phi^2\nu\|\pa_{y}U\pa_yf\|_2\|\pa_{yy}f\|_2 + \beta\phi^2k\Re\int \pa_{y}UVf\overline{\pa_yf}dy\\
        \leq& \al\phi\ep^{1/2}\nu\|\pa_{yy}f\|_2^2 + \lf(\frac{\beta^2}{4\al\gamma}\rg) \gamma\phi^3\ep^{-1/2}\nu\|\pa_{y}U\pa_yf\|_2^2 +  \beta\phi^2k\Re\int \pa_{y}UVf\overline{\pa_yf}dy.\label{T_beta_3}
    \end{align}
    Finally we estimate the term $\mathbb{T}_{\beta;4}$ in~\eqref{T_beta1234}
    \begin{align}
        \mathbb{T}_{\beta;4} =& \beta\phi^2\Re\int i\pa_{y}Uf\overline{\lf(\nu\pa_y^3f - ik\pa_yVf - ikV\pa_yf\rg)}dy\\
        =& \beta\phi^2\lf(-\nu\Re\int i\lf(\pa_{yy}Uf + \pa_{y}U\pa_yf\rg)\overline{\pa_{yy}f}dy - k\Re\int (\pa_{y}U\pa_yV)|f|^2dy - k\Re\int \pa_{y}UVf\overline{\pa_yf}dy\rg).
    \end{align}
    Now, we invoke the assumption \eqref{asmpt0} and the equivalence relation \eqref{equiv_UV} to obtain that
    \begin{align}
        \mathbb{T}_{\beta,4}\leq& \beta\phi^2\nu\|\pa_{yy}U\|_\infty\|f\|_2\|\pa_y^2f\|_2 + \beta\phi^2\nu\|\pa_{y}U\pa_yf\|_2\|\pa_y^2f\|_2 - \frac{\beta\phi^2|k|}{C_\ast}\Re\int |\pa_{y}U|^2|f|^2 dy - \beta\phi^2k\Re\int \pa_{y}UVf\overline{\pa_yf}dy\\
        \leq&  {\al} \phi\ep^{1/2}\nu\|\pa_y^2f\|_2^2 + \frac{\beta^2}{2\al}\phi^3\ep^{-1/2}\nu\|\pa_{yy}U\|_\infty^2\|f\|_2^2 + \lf(\frac{\beta^2}{2\al\gamma}\rg)\gamma\phi^3\ep^{-1/2}\nu\|\pa_{y}U\pa_yf\|_2^2\\
        & - \frac{\beta\phi^2|k|}{C_\ast}\|\pa_{y}Uf\|_2^2 - \beta\phi^2k\Re\int \pa_{y}UVf\overline{\pa_yf}dy.\label{T_beta_4}
    \end{align}
    Combining the estimates, we have
    \begin{align}
        \mathbb{T}_\beta \leq& \lf(\frac{1}{4}+ 4\beta C_\ast \rg)\nu\|\pa_yf\|_2^2 - \frac{3}{4}\frac{\beta\phi^2|k|}{C_\ast}\|\pa_{y}Uf\|_2^2+\lf(\frac{\beta\phi^2}{|k|^{1/2}}+\frac{\beta}{2\al}\phi\|\pa_{yy}U\|_\infty^2\rg)\beta \phi^2\nu^{1/2}|k|^{1/2}\|f\|_2^2 \\
        &+2{\al} \phi\ep^{1/2}\nu\|\pa_y^2f\|_2^2  + \lf(\frac{3\beta^2}{4\al\gamma}\rg)\gamma\phi^3\ep^{-1/2}\nu\|\pa_{y}U\pa_yf\|_2^2 .
    \end{align}
    This is the estimate \eqref{ndeg_beta_est}.
\end{proof}
\begin{proof}[Proof of Lemma~\ref{lem:nondegenerate gamma}]
    Combining the equation~\eqref{k_by_k_eq}, the smallness assumption \eqref{asmpt0}, and integration by parts yields the following bound
    \begin{align}
        \mathbb{T}_\gamma 
        \leq& 3\gamma\phi^2|k|\|\pa_{y}Uf\|_2^2 + 2\gamma\phi^3\ep^{-1/2}\lf(\int |\pa_{ty}U||\pa_{y}U||f|^2 dy + \Re\int |\pa_{y}U|^2\lf(\nu\pa_{yy}f - iVkf\rg)\overline{f}dy\rg)\\
        \leq& 3\gamma\phi^2|k|\|\pa_{y}Uf\|_2^2 +  2\gamma\phi^3\ep^{-1/2}\lf( \nu^{3/4} \|f\|_2\|\pa_{y}U f\|_2 -2\nu\Re\int \pa_{y}U\pa_yf\ \overline{\pa_{yy}U f}dy - \nu\| \pa_{y}U\pa_yf\|_2^2 \rg)\\
        \leq& \lf(\frac{3\gamma C_\ast}{\beta}+\frac{1}{4}\rg)\frac{\beta\phi^2|k|}{C_\ast}\|\pa_{y}Uf\|_2^2 + \lf(\frac{4C_\ast\gamma^2}{\beta^2|k|^{1/2}}\phi^2+\frac{4\gamma\phi}{\beta}\|\pa_{yy}U\|_\infty^2\|\rg)\beta\phi^2\nu^{1/2}|k|^{1/2}\|f\|_2^2  -\gamma\phi^3\ep^{-1/2}\nu\|\pa_{y}U\pa_yf\|_2^2.
    \end{align}This is \eqref{ndeg_gamma_est}.
\end{proof}


\section{Taylor Dispersion}\label{Taylor}
In this section, we prove Theorem \ref{Taylor_thm}. Let us focus on the Taylor dispersion regime $0<|k|\leq  \nu$ and further divide the proof into three steps.

\noindent
{\bf Step \# 1: Preliminaries.} 

Without loss of generality, we assume that $k> 0$ and focus on the hypoelliptic equation $\eqref{k_by_k_eq}_{\sigma=0}$. Thanks to the Dirichlet boundary condition $f(t,y=\pm 1)=0$, we have the following boundary constraints for \eqref{k_by_k_eq}:
\begin{align}
\wh f_k(t,y=\pm 1)=0, \quad \pa_{yy}\wh f_{k}(t, y=\pm 1)=0,\quad \forall t\geq0.\label{bc}
\end{align}
These boundary conditions enable us to implement integration by parts without creating boundary terms. We will adopt the simplified notation $f=\wh f_k$ in the remaining part of the section. 

Next, we identify the parameter regime in which the functional $\mathcal{T}$ \eqref{hypo_Tay} is comparable to the $H^1$-norm of the solution. We observe that 
\begin{align*}
\mathcal{T}[f]\geq \|f\|_2^2+\al\zeta \|\pa_y f\|_2^2- \beta\zeta\|\pa_y V\|_\infty\|f\|_2\|\pa_yf \|_2\geq \frac{1}{2}\|f\|_2^2+\al\zeta\lf(1-\frac{\beta^2}{\al}\|\pa_yV\|_{\infty}^2\rg)\|\pa_yf\|_2^2.
\end{align*}
Similar estimate holds for the upper bound.  
Hence if 
\begin{align}
\frac{\beta^2}{\al}\|\pa_yV\|_\infty^2\leq \frac{1}{2},\label{cnstr_ab_Taylor}
\end{align}
the following equivalence relation holds
\begin{align}
\frac{1}{2}\|f\|_2^2+\frac{1}{2}\al\zeta\|\pa_yf\|_2^2\leq \mathcal{T}[f]\leq
\frac{3}{2}\|f\|_2^2+\frac{3}{2}\al\zeta\|\pa_yf\|_2^2.   \label{Tay_equiv}
\end{align}
\ifx we consider the following hypocoercivity functional:
\begin{align}
T[f]=\|f\|_2^2+\al\|\pa_y f\|_2^2+\beta\Re\lan i\pa_{y}U\text{sign}{k} f,\pa_y f\ran.
\end{align}\fi

\noindent
{\bf Step \# 2: Hypocoercivity.} 
We compute the time derivative of the hypocoercivity functional $\mathcal{T}[f]$  \eqref{hypo_Tay}
\begin{align}\label{Tay_123}
\frac{d}{dt}\mathcal T[f]=\frac{d}{dt}\|f\|_2^2+\al\frac{d}{dt}\lf(\zeta\|\pa_y f\|_2^2\rg)+\beta \frac{d}{dt}\lf(\zeta\Re\lan i\pa_yV\  \text{sign}{k}\  f,\pa_y f\ran\rg)=:T_1+T_2+T_3.
\end{align}
By recalling the boundary condition \eqref{bc}, we apply integration by parts to obtain
\begin{align}
T_1=-2\nu\|\pa_y f\|_2^2.\label{Tay_1}
\end{align}
The $T_2$-term can be estimated using the constraint $|k|/\nu\leq1,$ integration by parts (with boundary condition \eqref{bc}), H\"older inequality and Young's inequality as follows
\begin{align}
T_2=&\al\zeta'\|\pa_y f\|_2^2+2\al\zeta\Re\lan\nu\pa_{yyy}f -\pa_y Vik f-Vik\pa_y f ,\pa_y f\ran \\
\leq& \al\frac{|k|^2}{\nu^2}\nu\|\pa_y f\|_2^2 -2\al\zeta\nu\|\pa_{yy}f\|_2^2+2\al \zeta\frac{|k|}{\nu}\nu\|\pa_y Vf\|_2\|\pa_yf\|_2\\
\leq &\al \nu\|\pa_y f\|_2^2 -2\al\zeta\nu\|\pa_{yy}f\|_2^2+\frac{\beta}{4}\zeta |k|\|\pa_y Vf\|_2^2+ \frac{4\al^2}{\beta}\zeta \nu\|\pa_yf\|_2^2.\label{Tay_2}
\end{align}
The $T_3$-term in \eqref{Tay_123} can be estimated in a similar fashion as the one in \eqref{T_beta1234}. We decompose the term as follows
\begin{align}
T_3=&\zeta' \beta\Re\int i\pa_y V f \overline{\pa_y f}dy+\zeta\beta\Re\int i \pa_{ty}V f\overline{\pa_y f}dy+\beta\zeta\Re\int i \pa_yV  (\nu\pa_{yy}f-ikV f)\overline{\pa_y f}dy\\
&+\beta\zeta\Re\int i \pa_yV f \overline{ (\nu\pa_{yyy}f-ik\pa_y V f-ik V\pa_y f)}dy=:\sum_{i=1}^4 T_{3i}.\label{T_3_Tay}
\end{align}
We recall the assumption \eqref{Taylorcnstr}  and the definition \eqref{ep&t_weights} to obtain\begin{align}\label{T_312_Tay}
T_{31}+T_{32}\leq &\nu^{-1}|k|^2\beta\|\pa_y V\|_{\infty}\|f \|_2\|\pa_y f\|_2+\beta\mathfrak{C}_2\nu\|f\|_2\|\pa_yf\|_2.
\end{align}
Similar to the estimate in \eqref{T_beta_3}, we estimate the $T_{33}$-term in \eqref{T_3_Tay} as follows
\begin{align}
T_{33}\leq& \beta\zeta\Re\int i\pa_{y}V\lf(\nu\pa_{yy}f-iVkf\rg)\overline{\pa_yf}dy = \beta\zeta\lf(\nu\Re\int i\pa_{y}V\pa_{yy}f\overline{\pa_yf}dy + k\Re\int V\pa_{y}Vf\overline{\pa_yf}dy\rg)\\
        \leq& \beta\zeta \nu \| \pa_y V\|_{\infty}\|\pa_{yy} f\|_2\|\pa_y f\|_2 + \beta\zeta k\Re\int V\pa_{y}Vf\overline{\pa_yf}dy.\label{T_33_Tay}
\end{align} Similar to the estimate in 
\eqref{T_beta_4}, we invoke the boundary condition \eqref{bc} to estimate the $T_{34}$-term in \eqref{T_3_Tay} as follows
\begin{align}
T_{34}\leq& \beta\zeta\nu\|\pa_{yy}V\|_\infty\|f\|_2\|\pa_{yy}f\|_2 + \beta\zeta\nu\|\pa_{y}V\|_\infty\|\pa_yf\|_2\|\pa_{yy}f\|_2 - {\beta\zeta|k|}\|\pa_{y}Vf\|_2^2 - \beta\zeta k\Re\int V\pa_{y}Vf\overline{\pa_yf}dy.\\ \label{T_34_Tay}
\end{align}
Finally, we recall the Poincar\'e inequality for $f\in H^1_0([-1,1])$:
\begin{align} \label{Poincare}
\|f\|_2\leq C\|\pa_y f\|_2.
\end{align}
Now we combine the decomposition \eqref{T_3_Tay}, the estimates \eqref{T_312_Tay}, \eqref{T_33_Tay} and \eqref{T_34_Tay} and the Poincar\'e inequality \eqref{Poincare} to estimate the $T_3$-term as follows
\begin{align}
T_3\leq& C_0\lf( \nu^{-2}|k|^2\|\pa_yV\|_\infty+ \mathfrak{C}_2+\frac{\beta}{\al}\|\pa_y V\|_\infty^2 +\frac{\beta}{\al}\|\pa_{yy} V\|_\infty^2 \rg)\beta \nu \|\pa_yf\|_2^2+\al\zeta\nu \|\pa_{yy}f\|_2^2-\frac12\beta\zeta|k|\lf\|\pa_y Vf\rg\|_2^2.\\\label{Tay_3}
\end{align}
Here $C_0\geq 1$ is a universal constant. Now we set 
\begin{align}
\al=\beta=\frac{1}{16C_0\lf(1+\mathfrak{C}_2+\|\pa_y V\|_{L_{t,y}^\infty }^2+\|\pa_{yy}V\|_{L_{t,y}^\infty}^2\rg)},
\end{align}
which is consistent with \eqref{cnstr_ab_Taylor}. Combining the decomposition \eqref{Tay_123} and the estimates  \eqref{Tay_1}, \eqref{Tay_2}, \eqref{Tay_3}  yields that 
\begin{align}
\frac{d}{dt}\mathcal{T}[f]\leq-\nu\|\pa_y f\|_2^2-\frac{1}{4}\beta \zeta|k|\|\pa_yV f \|_2^2.
\end{align}
Thanks to the spectral inequality \eqref{Spec_Tay} and the equivalence \eqref{Tay_equiv}, we obtain that for $|k|\leq \nu$, 
\begin{align} 
\frac{d}{dt}\mathcal{T}[f]\leq&-\frac{1}{2}\nu\|\pa_y f\|_2^2 -\frac{\beta}{4}\frac{|k|^2}{\nu}\lf(\frac{\nu^2}{|k|^{2}}\|\pa_y f\|_2^2+\frac{\nu}{|k|}\zeta\|\pa_y V f\|_2^2\rg)\\
\leq&-\frac{1}{2}\frac{k^2}{\nu^2}\nu\|\pa_y f\|_2^2 -\frac{|k|^2}{4\nu}\beta\zeta(\|\pa_{y}f\|_{2}^2+\|\pa_y V f\|_2^2)\leq -\frac{1}{2}\frac{|k|^2}{\nu}\|\pa_yf\|_2^2-\frac{|k|^2}{\nu}\frac{\beta\zeta}{4\mathcal{C}_{spec}}\|f\|_2^2\\
\leq&-\frac{|k|^2}{\nu}\frac{\beta\zeta}{4\mathcal{C}_{spec}}\lf(\|f\|_2^2+\al\zeta\|\pa_yf\|_2^2\rg)\leq -\frac{|k|^2}{\nu}\frac{\beta\zeta}{6\mathcal{C}_{spec}}\mathcal{T}[f]=:-\delta(\beta,\mathcal{C}_{spec}^{-1})\zeta\frac{|k|^2}{\nu}\mathcal{T}[f].\label{dtT}
\end{align}

\noindent 
{\bf Step \# 3: Conclusion.} Without loss of generality, set $T\geq2\nu|k|^{-2}$. 
We observe that for $t\in[0,\nu|k|^{-2}]$, since $\frac{d}{dt}\mathcal{T}\leq 0,$ 
\begin{align}
\mathcal{T}[f(t)]\leq 
\mathcal{T}[f(0)]=\|f(0)\|_{2}^2. 
\end{align}
Thanks to the estimate \eqref{dtT}, we have that 
\begin{align}
\mathcal{T}[f(t)]\leq 
\mathcal{T}[f(\nu|k|^{-2})]\exp\lf\{-\delta \frac{|k|^2}{\nu}(t-\nu|k|^{-2})\rg\}\leq e\|f(0)\|_{2}^2\exp\lf\{-\delta \frac{|k|^2}{\nu} t \rg\},\quad \forall t\in[\nu|k|^{-2},T]. 
\end{align}
This concludes the proof of \eqref{Hypo_est_Taylor} and Theorem \ref{Taylor_thm}.
\ifx
The time weight associated with the Taylor dispersion is 
\begin{align}
\zeta(t)=\min\lf\{1,\frac{|k|^2}{\nu}t\rg\}.
\end{align}
\begin{align}
T[f]=\|f\|_2^2+\al\zeta\|\pa_y f\|_2^2+\beta\zeta\Re\lan i\pa_{y}U\text{sign}{k} f,\pa_y f\ran.
\end{align}
Now the time derivative of the $\al$ term reads
\begin{align}
\frac{|k|^2}{\nu}\al\|\pa_y f\|_2^2=
\frac{|k|^2}{\nu^2}\nu\al\|\pa_y f\|_2^2\leq \al\nu\|\pa_y f\|_2^2.
\end{align}
Now the time derivative of the $beta$ term reads as follows
\begin{align*}
\frac{|k|^2}{\nu^2}\nu \beta\Re\lan i\pa_{y}U f,\pa_y f\ran\leq \beta\nu\|\pa_{y}U\|_\infty\|f\|_2\|\pa_y f\|_2.
\end{align*}\fi

\appendix
\section{Technical Lemmas}

The proof makes use of several spectral inequalities. We present them below. 
\begin{lem} 
a) Consider the domain, i.e., $y\in\Torus$. Assume that $U(t,y)$ has $N$ nondegenerate critical points $\{y_i(t)\}_{i=1}^N$ for $t\in[0,T]$. Moreover, there exist $N$ open neighbourhoods $B_{r}(y_i(t)),\, i=1,\cdots, N$, such that 
    \begin{align}
        | \pa_y{U}(t,y)|^2 \geq& \mathfrak{C}_0^{-1} (y-y_i(t))^2,\quad \forall t\in[0,T],\quad \forall y\in B_{r}(y_i(t)),\quad \forall y_i(t)\in \lf\{y\big|{\pa_y} U(t,y)=0\rg\},\\
        |\pa_yU(t,y)|&\in[\mathfrak C_1^{-1},\mathfrak C_1],\quad \forall y\in (\cup_{i=1}^NB_{r}(y_i(t))^c.
    \end{align}
    Then for $\nu$ small enough depending on the shear $U$, there exists a constant $\mathfrak{C}_{\text{Spec}}\geq 1$ such that the following estimate hold ($\ep=\nu/|k|$)
    \begin{align}\label{spectral}
        \ep^{1/2}\|f\|_{L^2(\Torus)}^2\leq \ep\|\pa_y f\|_{L^2(\Torus)}^2+\mathfrak{C}_{\text{Spec}}\lf\|  \pa_yU(t,\cdot)f\rg\|_{L^2(\Torus)}^2. 
    \end{align}
       
\noindent    
    b) Consider functions $f\in H_0^1([-1,1])$. Assume that the shear flow $V(t,y)$ satisfies the conditions specified in Theorem \ref{Taylor_thm}. Then there exists a constant $\mathcal{C}_{Spec}=\mathcal{C}_{Spec}(m_0, \mathfrak{C}_3, r_0)\ge 1$ such that the following estimate holds:
\begin{align}\label{Spec_Tay}
\|f\|_{L^2([-1,1])}^2\leq \mathcal{C}_{Spec}\|\pa_y f\|_{L^2([-1,1])}^2+\mathcal{C}_{Spec}\|\pa_y V(t,\cdot) f\|_{L^2([-1,1])}^2.
\end{align}
\end{lem}
\begin{proof}a) The proof of the theorem is stated in the paper \cite{BCZ15}. For the sake of completeness, we provide a different proof here. We can apply a partition of unity $\{\chi_i\}_{i=0}^N$ to decompose the function $f=f(\chi_0+\sum_{i=1}^n\chi_i)$, where $\{\chi_i\}_{i\neq 0}$ are supported near the critical points $y_i(t)$ and $\chi_0$ is supported away from the critical points. Moreover, $\sum_{i=0}^n\|\pa_y \chi_i\|_\infty\leq C$ and the supports of $\{\chi_i\}_{i\neq 0}$ are pairwise disjoint. 
Now we use the integration by parts formula
\begin{align}\n
    \ep^{1/2}\int_\rr |f_i|^2 dy=& \frac{1}{2}\ep^{1/2}\bigg|\int_{\rr} |f_i|^2 \frac{d^2}{dy^2}(y-y_i)^2 dy\bigg|  = \ep^{1/2}\bigg|\int_\rr \pa_y|f_i|^2 (y-y_i) dy\bigg|\\
    \leq& {2}\mathfrak{C}_0\ep^{1/2}\bigg|\Re\int_\rr \overline{f_i}\pa_y f_i   |\pa_yU|  dy\bigg|\leq \frac{1}{2}\ep\|\pa_y f_i\|_{L^2(\rr)}^2+C(\mathfrak C_0)\lf\|\pa_yU f_i\rg\|_{L^2(\rr)}^2,\quad i\neq 0.\label{spec_nr_crit}
\end{align}
Since the supports of the cutoff functions $\chi_i,\, i\neq 0$ are disjoint, we have that 
\begin{align}
\ep^{1/2}\int_{\Torus} |f(1-\chi_0)|^2dy\leq \ep \|\pa_y (f(1-\chi_0))\|_{L^2}^2+C(\mathfrak{C}_0)\|\pa_y U f(1-\chi_0)\|_{L^2}^2.
\end{align}
We further observe that, since the $|\pa_yU|\geq c>0$ on the support of $\chi_0$, 
\begin{align}
    \ep^{1/2}\|f\chi_0\|_{L^2}^2\leq C\||\pa_yU|f\chi_0\|_{L^2}^2.
\end{align}
Combining the above estimates, we have that 
\begin{align}\n
    \ep^{1/2}\|f\|_{L^2}^2\leq& 2\ep^{1/2}\|f\chi_{0}\|_{L^2}^2+2\ep^{1/2}\|f(1-\chi_0)\|_{L^2}^2\leq \ep\|\pa_y(f(1-\chi_0))\|_{L^2}^2+C(\mathfrak{C}_0)\lf\||\pa_yU| f\rg\|_{L^2}^2\\
\leq&\ep\|\pa_yf\|_{L^2}^2+C(\mathfrak{C}_0)\lf\||\pa_yU| f\rg\|_{L^2}^2+\ep\|\pa_y\chi_0\|_{L^\infty}^2\|f\|_{L^2}^2.
\end{align}
We can take the $\nu$ small enough so that the left-hand side absorbs the last term. This concludes the proof of the lemma.

b) The proof of the estimate \eqref{Spec_Tay} is similar to the proof above. We decompose the function $f$ as follows
\begin{align}
f=\sum_{i=0}^{N(t)}f\chi_i,\quad \text{support }\chi_i(t,\cdot)\in U_i:=B_{r_0/2}(y_i(t))\cap [-1,1], \quad \chi_i\in C^\infty(\rr_+\times[-1,1]), \quad \forall i\in\{0,1,\cdots, N(t)\}.
\end{align}
We observe that the sets $U_i$ are pairwise disjoint and $\text{distance}(U_i,U_j)\geq r_0>0$ for $i\neq j,\, i,j\in \{1, 2,\cdots, N(t)\} $. We further assume that $\chi_i(t,y)\equiv 1$ for $y\in B_{r_0/4}(y_i(t))\cap [-1,1]$. As a result of this choice and the condition \eqref{Taylorcnstr}, there exists a positive constant $c$, which depends on the shear flow $V$, such that 
\begin{align}\label{non_deg_Tay}
|\pa_y V(t,y)|\geq c>0,\quad \forall y\in \text{support} \{\pa_y \chi_i(t,\cdot)\}. 
\end{align}
We apply the Poincar\'e inequality and the observation \eqref{non_deg_Tay} to obtain that for $i\in \{1,\cdots, N(t)\}$, 
\begin{align}
\|f \chi_i\|_{L^2([-1,1])}\leq& C\|\pa_y (f\chi_i)\|_{L^2([-1,1])}\leq C\|\pa_y f\|_{L^2([-1,1])}+C\|f \pa_y \chi_i(t,\cdot)\|_{L^2([-1,1])}\\
\leq& C\|\pa_y f\|_{L^2([-1,1])}+C\|f\pa_y V(t,\cdot)\|_{L^2([-1,1])}.
\end{align}
Moreover, we have that $\|f\chi_0\|_{L^2([-1,1])}\leq C\|f\pa_yV(t,\cdot)\|_{L^2([-1,1])}$ because the $\pa_y V$ is non-vanishing on the support of $\chi_0$. Through summing all the contributions, we obtain \eqref{Spec_Tay}.
\end{proof}

\bibliographystyle{abbrv} 
\bibliography{SimingBib,nonlocal_eqns,JacobBib}

\def\cprime{$'$}
\begin{bibdiv}
\begin{biblist}

\bib{AlbrittonBeekieNovack21}{article}{
      author={Albritton, Dallas},
      author={Beekie, Rajendra},
      author={Novack, Matthew},
       title={Enhanced dissipation and {H\"ormander's} hypoellipticity},
        date={2022},
        ISSN={0022-1236},
     journal={J. Funct. Anal.},
      volume={283},
      number={3},
       pages={Paper No. 109522, 38},
         url={https://doi.org/10.1016/j.jfa.2022.109522},
      review={\MR{4413303}},
}

\bib{AlbrittonOhm22}{article}{
      author={Albritton, Dallas},
      author={Ohm, Laurel},
       title={On the stabilizing effect of swimming in an active suspension},
        date={2022},
     journal={arXiv:2205.04922},
         url={https://doi.org/10.48550/arXiv.2205.04922},
}

\bib{Aris56}{article}{
      author={Aris, Rutherford},
       title={On the dispersion of a solute in a fluid flowing through a tube},
        date={1956},
     journal={Proceedings of the Royal Society of London. Series A.
  Mathematical and Physical Sciences},
      volume={235},
      number={1200},
       pages={67\ndash 77},
}

\bib{BeckChaudharyWayne20}{article}{
      author={Beck, Margaret},
      author={Chaudhary, Osman},
      author={Wayne, C~Eugene},
       title={Rigorous justification of taylor dispersion via center manifolds
  and hypocoercivity},
        date={2020},
     journal={Archive for Rational Mechanics and Analysis},
      volume={235},
      number={2},
       pages={1105\ndash 1149},
}

\bib{BCZ15}{article}{
      author={Bedrossian, J.},
      author={Coti~Zelati, M.},
       title={Enhanced dissipation, hypoellipticity, and anomalous small noise
  inviscid limits in shear flows},
        date={2017},
        ISSN={0003-9527},
     journal={Arch. Ration. Mech. Anal.},
      volume={224},
      number={3},
       pages={1161\ndash 1204},
         url={https://doi.org/10.1007/s00205-017-1099-y},
      review={\MR{3621820}},
}

\bib{BGM15II}{article}{
      author={Bedrossian, J.},
      author={Germain, P.},
      author={Masmoudi, N.},
       title={Dynamics near the subcritical transition of the {3D Couette flow
  II: Above} threshold},
        date={2015},
     journal={arXiv:1506.03721},
}

\bib{BGM15III}{article}{
      author={Bedrossian, J.},
      author={Germain, P.},
      author={Masmoudi, N.},
       title={On the stability threshold for the {3D Couette} flow in {Sobolev}
  regularity},
        date={2017},
        ISSN={0003-486X},
     journal={Ann. of Math. (2)},
      volume={185},
      number={2},
       pages={541\ndash 608},
         url={https://doi.org/10.4007/annals.2017.185.2.4},
      review={\MR{3612004}},
}

\bib{BGM15I}{article}{
      author={Bedrossian, J.},
      author={Germain, P.},
      author={Masmoudi, N.},
       title={Dynamics near the subcritical transition of the {3D Couette flow
  I: Below} threshold},
        date={2020},
        ISSN={0065-9266},
     journal={Mem. Amer. Math. Soc.},
      volume={266},
      number={1294},
       pages={v+158},
         url={https://doi.org/10.1090/memo/1294},
      review={\MR{4126259}},
}

\bib{BMV14}{article}{
      author={Bedrossian, J.},
      author={Masmoudi, N.},
      author={Vicol, V.},
       title={Enhanced dissipation and inviscid damping in the inviscid limit
  of the {Navier-Stokes} equations near the {2D Couette} flow},
        date={2016},
     journal={Arch. Rat. Mech. Anal.},
      volume={216},
      number={3},
       pages={1087\ndash 1159},
}

\bib{Bedrossian17}{article}{
      author={Bedrossian, Jacob},
       title={Suppression of plasma echoes and {L}andau damping in {S}obolev
  spaces by weak collisions in a {V}lasov-{F}okker-{P}lanck equation},
        date={2017},
     journal={Ann. PDE},
      volume={3},
      number={2},
       pages={Paper No. 19, 66},
}

\bib{BedrossianBlumenthalPunshonSmith192}{article}{
      author={Bedrossian, Jacob},
      author={Blumenthal, Alex},
      author={Punshon-Smith, Sam},
       title={The {Batchelor} spectrum of passive scalar turbulence in
  stochastic fluid mechanics},
        date={2019},
     journal={arXiv:1911.11014},
}

\bib{BedrossianHe19}{article}{
      author={Bedrossian, Jacob},
      author={He, Siming},
       title={Inviscid damping and enhanced dissipation of the boundary layer
  for 2d {Navier-Stokes} linearized around couette flow in a channel},
     journal={arXiv:1909.07230},
}

\bib{BedrossianHe16}{article}{
      author={Bedrossian, Jacob},
      author={He, Siming},
       title={Suppression of blow-up in {Patlak-Keller-Segel} via shear flows},
        date={2018},
     journal={SIAM Journal on Mathematical Analysis},
      volume={50},
      number={6},
       pages={6365\ndash 6372},
}

\bib{BedrossianCotiZelatiDolce22}{article}{
      author={Bedrossian, Jacob},
      author={Zelati, Michele~Coti},
      author={Dolce, Michele},
       title={{Taylor} dispersion and phase mixing in the non-cutoff
  {Boltzmann} equation on the whole space},
     journal={arXiv:2211.05079},
         url={https://doi.org/10.48550/arXiv.2211.05079},
}

\bib{ChenLiWeiZhang18}{article}{
      author={Chen, Qi},
      author={Li, Te},
      author={Wei, Dongyi},
      author={Zhang, Zhifei},
       title={Transition threshold for the 2-{D} {C}ouette flow in a finite
  channel},
        date={2020},
        ISSN={0003-9527},
     journal={Arch. Ration. Mech. Anal.},
      volume={238},
      number={1},
       pages={125\ndash 183},
         url={https://doi.org/10.1007/s00205-020-01538-y},
      review={\MR{4121130}},
}

\bib{CKRZ08}{article}{
      author={Constantin, P.},
      author={Kiselev, A.},
      author={Ryzhik, L.},
      author={Zlato{\v{s}}, A.},
       title={Diffusion and mixing in fluid flow},
        date={2008},
     journal={Ann. of Math. (2)},
      volume={168},
       pages={643\ndash 674},
}

\bib{CKR00}{article}{
      author={Constantin, Peter},
      author={Kiselev, Alexander},
      author={Ryzhik, Leonid},
       title={Quenching of flames by fluid advection},
        date={2001},
     journal={Comm. Pure Appl. Math.},
      volume={54},
      number={11},
       pages={1320\ndash 1342},
}

\bib{ElgindiCotiZelatiDelgadino18}{article}{
      author={Coti~Zelati, Michele},
      author={Delgadino, Matias~G.},
      author={Elgindi, Tarek~M.},
       title={On the relation between enhanced dissipation timescales and
  mixing rates},
        date={2020},
        ISSN={0010-3640},
     journal={Comm. Pure Appl. Math.},
      volume={73},
      number={6},
       pages={1205\ndash 1244},
         url={https://doi.org/10.1002/cpa.21831},
      review={\MR{4156602}},
}

\bib{CotiZelatiDolce20}{article}{
      author={Coti~Zelati, Michele},
      author={Dolce, Michele},
       title={Separation of time-scales in drift-diffusion equations on
  {$\Bbb{R}^2$}},
        date={2020},
        ISSN={0021-7824},
     journal={J. Math. Pures Appl. (9)},
      volume={142},
       pages={58\ndash 75},
         url={https://doi.org/10.1016/j.matpur.2020.08.001},
      review={\MR{4149684}},
}

\bib{CotiZelatiDolceFengMazzucato}{article}{
      author={Coti-Zelati, Michele},
      author={Dolce, Michele},
      author={Feng, Yuanyuan},
      author={Mazzucato, Anna~L.},
       title={Global existence for the two-dimensional {Kuramoto-Sivashinsky}
  equation with a shear flow},
        date={2021},
     journal={arXiv:2103.02971},
}

\bib{CotiZelatiDrivas19}{article}{
      author={Coti-Zelati, Michele},
      author={Drivas, Theodore~D.},
       title={A stochastic approach to enhanced diffusion},
        date={2019},
     journal={arXiv:1911.09995v1},
}

\bib{FengFengIyerThiffeault20}{article}{
      author={Feng, Yu},
      author={Feng, Yuanyuan},
      author={Iyer, Gautam},
      author={Thiffeault, Jean-Luc},
       title={Phase separation in the advective {C}ahn-{H}illiard equation},
        date={2020},
        ISSN={0938-8974},
     journal={J. Nonlinear Sci.},
      volume={30},
      number={6},
       pages={2821\ndash 2845},
         url={https://doi.org/10.1007/s00332-020-09637-6},
      review={\MR{4170312}},
}

\bib{FengIyer19}{article}{
      author={Feng, Yuanyuan},
      author={Iyer, Gautam},
       title={Dissipation enhancement by mixing},
        date={2019},
     journal={Nonlinearity},
      volume={32},
      number={5},
       pages={1810\ndash 1851},
}

\bib{FengShiWang22}{article}{
      author={Feng, Yuanyuan},
      author={Shi, Binbin},
      author={Wang, Weike},
       title={Dissipation enhancement of planar helical flows and applications
  to three-dimensional kuramoto-sivashinsky and keller-segel equations},
        date={2022},
     journal={Journal of Differential Equations},
      volume={313},
       pages={420\ndash 449},
}

\bib{CotiZelatiGallay21}{article}{
      author={Gallay, Thierry},
      author={Zelati, Michele~Coti},
       title={Enhanced dissipation and taylor dispersion in higher-dimensional
  parallel shear flows},
        date={2021},
     journal={arXiv preprint arXiv:2108.11192},
}

\bib{GongHeKiselev21}{article}{
      author={Gong, Yishu},
      author={He, Siming},
      author={Kiselev, Alexander},
       title={Random search in fluid flow aided by chemotaxis},
        date={2021},
     journal={arXiv:2107.02913},
}

\bib{He}{article}{
      author={He, S.},
       title={Suppression of blow-up in parabolic-parabolic
  {Patlak-Keller-Segel} via strictly monotone shear flows},
        date={2018},
     journal={Nonlinearity},
      volume={31},
      number={8},
       pages={3651\ndash 3688},
}

\bib{He21}{article}{
      author={He, Siming},
       title={Enhanced dissipation, hypoellipticity for passive scalar
  equations with fractional dissipation},
        date={2022},
        ISSN={0022-1236},
     journal={J. Funct. Anal.},
      volume={282},
      number={3},
       pages={Paper No. 109319, 28},
         url={https://doi.org/10.1016/j.jfa.2021.109319},
      review={\MR{4343756}},
}

\bib{HeKiselev21}{article}{
      author={He, Siming},
      author={Kiselev, Alexander},
       title={Stirring speeds up chemical reaction},
        date={2022},
     journal={Nonlinearity},
      volume={35},
      number={8},
       pages={4599},
}

\bib{HeTadmorZlatos}{article}{
      author={He, Siming},
      author={Tadmor, Eitan},
      author={Zlato{\v{s}}, Andrej},
       title={On the fast spreading scenario},
        date={2022},
     journal={Communications of the American Mathematical Society},
      volume={2},
      number={04},
       pages={149\ndash 171},
}

\bib{HuKiselev23}{article}{
      author={Hu, Zhongtian},
      author={Kiselev, Alexander},
       title={Suppression of chemotactic blowup by strong buoyancy in
  stokes-boussinesq flow with cold boundary},
        date={2023},
     journal={arXiv preprint arXiv:2309.04349},
}

\bib{HuKiselevYao23}{article}{
      author={Hu, Zhongtian},
      author={Kiselev, Alexander},
      author={Yao, Yao},
       title={Suppression of chemotactic singularity by buoyancy},
        date={2023},
     journal={arXiv preprint arXiv:2305.01036},
}

\bib{IyerXuZlatos}{article}{
      author={Iyer, Gautam},
      author={Xu, Xiaoqian},
      author={Zlato\v{s}, Andrej},
       title={Convection-induced singularity suppression in the
  {K}eller-{S}egel and other non-linear {PDE}s},
        date={2021},
     journal={Trans. Amer. Math. Soc.},
      volume={374},
      number={9},
       pages={6039\ndash 6058},
}

\bib{Kelvin87}{article}{
      author={Kelvin, Lord},
       title={Stability of fluid motion-rectilinear motion of viscous fluid
  between two parallel plates},
        date={1887},
     journal={Phil. Mag.},
      number={24},
       pages={188},
}

\bib{KiselevXu15}{article}{
      author={Kiselev, A.},
      author={Xu, X.},
       title={Suppression of chemotactic explosion by mixing},
        date={2016},
     journal={Arch. Ration. Mech. Anal.},
      volume={222},
      number={2},
       pages={1077\ndash 1112},
}

\bib{KiselevZlatos}{article}{
      author={Kiselev, Alexander},
      author={Zlato{s}, Andrej},
       title={Quenching of combustion by shear flows},
        date={2006},
     journal={Duke Math. J.},
      volume={132},
      number={1},
       pages={49\ndash 72},
}

\bib{LiZhao21}{article}{
      author={Li, Hui},
      author={Zhao, Weiren},
       title={Metastability for the dissipative quasi-geostrophic equation and
  the non-local enhancement},
        date={2021},
     journal={arXiv:2107.10594},
}

\bib{Taylor53}{article}{
      author={Taylor, Geoffrey~Ingram},
       title={Dispersion of soluble matter in solvent flowing slowly through a
  tube},
        date={1953},
     journal={Proceedings of the Royal Society of London. Series A.
  Mathematical and Physical Sciences},
      volume={219},
      number={1137},
       pages={186\ndash 203},
}

\bib{Wei18}{article}{
      author={Wei, Dongyi},
       title={Diffusion and mixing in fluid flow via the resolvent estimate},
        date={2019},
     journal={Science China Mathematics},
       pages={1\ndash 12},
}

\bib{WeiZhang19}{article}{
      author={Wei, Dongyi},
      author={Zhang, Zhifei},
       title={Enhanced dissipation for the {K}olmogorov flow via the
  hypocoercivity method},
        date={2019},
        ISSN={1674-7283},
     journal={Sci. China Math.},
      volume={62},
      number={6},
       pages={1219\ndash 1232},
         url={https://doi.org/10.1007/s11425-018-9508-5},
      review={\MR{3951889}},
}

\bib{YoungJones91}{article}{
      author={Young, WR~a},
      author={Jones, Scott},
       title={Shear dispersion},
        date={1991},
     journal={Physics of Fluids A: Fluid Dynamics},
      volume={3},
      number={5},
       pages={1087\ndash 1101},
}

\bib{CotiZelatiDietertGerardVaret22}{article}{
      author={Zelati, Michele~Coti},
      author={Dietert, Helge},
      author={G\'{e}rard-Varet, David},
       title={Orientation mixing in active suspensions},
     journal={arXiv:2207.08431},
         url={https://doi.org/10.48550/arXiv.2207.08431},
}

\bib{CotiZelatiDolceLo23}{article}{
      author={Zelati, Michele~Coti},
      author={Dolce, Michele},
      author={Lo, Chia-Chun},
       title={Diffusion enhancement and taylor dispersion for rotationally
  symmetric flows in discs and pipes},
        date={2023},
     journal={arXiv preprint arXiv:2305.18162},
}

\bib{Zlatos2010}{article}{
      author={Zlato{\v{s}}, A.},
       title={Diffusion in fluid flow: dissipation enhancement by flows in
  2{D}},
        date={2010},
     journal={Comm. Partial Differential Equations},
      volume={35},
      number={3},
       pages={496\ndash 534},
}

\bib{Zlatos11}{article}{
      author={Zlato{s}, Andrej},
       title={Reaction-diffusion front speed enhancement by flows},
        date={2011},
     journal={Ann. Inst. H. Poincar\'{e} Anal. Non Lin\'{e}aire},
      volume={28},
      number={5},
       pages={711\ndash 726},
}

\end{biblist}
\end{bibdiv}
\end{document}